\documentclass[11pt,times, 3p, reqno]{amsart}
\usepackage[T1]{fontenc}
\usepackage{a4wide}
\usepackage{booktabs}
\usepackage{hyperref}
\usepackage{enumerate,fullpage}
\usepackage{amssymb,amsmath,graphicx,amsthm}
\usepackage{color}

\usepackage{amsfonts,amsmath,amsthm,mathrsfs}
\usepackage{amssymb}
\usepackage{color}
\usepackage{enumerate}
\usepackage{subfloat,subfigure}
\usepackage{graphics,verbatim,url}
\usepackage{index}

\newcommand{\be}{\begin{eqnarray*}}
	\newcommand{\en}{\end{eqnarray*}}
\newcommand{\bes}{\begin{eqnarray}}
	\newcommand{\ens}{\end{eqnarray}}

\usepackage{epic, curves}
\def\nn{\nonumber}

\newcommand{\ep}{\epsilon}

\newcommand {\norm}[1] {\| #1 \|}

\newtheorem{theorem}{Theorem}[section]

\newtheorem{lemma}{Lemma}[section]

\newtheorem{remark}{Remark}[section]

\def\bq{\begin{equation}}
	\def\eq{\end{equation}}
\def\bqq{\begin{eqnarray*}}
	\def\eqq{\end{eqnarray*}}

\def\nn{\nonumber}

	\title[ 	Regularization of  sideways problem for  a  time fractional diffusion equation with nonlinear source ]{	Regularization of  sideways problem for  a  time fractional diffusion equation with nonlinear source  }

\author[T.B. Ngoc]{ Tran Bao Ngoc}
		\address[T.B. Ngoc]{Department of Mathematical Economics, Banking University of Ho Chi Minh City, Vietnam }
		\address{Department of Mathematics and Computer Science, VNUHCM-University of Science\\ 227 Nguyen Van Cu Str., Dist. 5, Ho Chi Minh City, Vietnam}
		\email{	 tranbaongoc@hcmuaf.edu.vn }

	\author[N.H. Tuan]{Nguyen Huy Tuan}
	\address[N.H. Tuan]{ Applied Analysis Research Group
		Faculty of Mathematics and Statistics
		Ton Duc Thang University
		Ho Chi Minh City, Viet Nam  }
	\email{nguyenhuytuan@tdt.edu.vn}

		\author[M. Kirane]{ Mokhtar Kirane}
	\address[M. Kirane] {Lasie, Facult\'e des Sciences et Technologies, Universit\'e de La Rochelle, Avenue M. Cr\'epeau, La Rochelle, Cedex 17042, France}
	
\begin{document}
		
		\begin{abstract}
			 In this paper, we consider an inverse problem for a time-fractional diffusion equation with a nonlinear source.  We prove that the considered problem is ill-posed, i.e. the solution does not depend continuously on the data. The problem is ill-posed in the sense of Hadamard. Under some weak {\color{black} a} priori assumptions on the sought solution, we propose a new regularization method for stabili{\color{black}z}ing the ill-posed problem.
			We also provide a numerical example to illustrate our results. 
		\end{abstract}
		
		\maketitle
	

	\section{Introduction} \thispagestyle{empty}

		In this article, we consider the following concentration identification problem (CIP) for
	the time fractional nonlinear diffusion equation:
	\begin{align}
		D^\alpha_t u(x,t) - u_{xx}(x,t)& =f(x,t,u(x,t)), \qquad x > 0,\quad t>0, \qquad  0<\alpha<1, \label{problem1}
	\end{align}
	with the Cauchy condition and initial condition
	\begin{align}
		u(x,0)&=0, \hspace*{1.65cm} x \ge 0, \label{problem2}\\
		u(0,t)&=g(t), \hspace*{1.3 cm}  t \ge 0, \label{problem3}\\
		u_x(0,t)&=h(t), \hspace*{1.3cm} t \ge 0, \label{problem4}
	\end{align}
	where  $ u$ is the solute concentration (see \cite{zheng}), $f $ is the  source term   defined later. 
	The function  $g(t)$ and $h(t)$ denote  the solute concentration and the measurement datum of dispersion flux, respectively, on {\color{black} the} left boundary.  We will recover the solute concentration
	$u(x, t)$  in the region  $\{ (x,t),~0\le x < 1, t>0 \}$  from the measurement data of source
	terms $ f (x,t,u(x,t))$ and boundary concentrations $g(t),~h(t)$.  The fractional derivative $\displaystyle \frac{\partial ^\alpha u}{\partial t^\alpha}$ is the Caputo fractional derivative of order $\alpha$  defined by  \cite{Podlubny}
	\begin{align}
		\frac{\partial ^\alpha u}{\partial t^\alpha}(t)& =  \frac{1}{\Gamma(1-\alpha)} \int_0^t \frac{\partial  u(x,s)}{\partial s} \frac{ds}{(t-s)^\alpha} \quad \textrm{ for } \quad 0< \alpha < 1,\\
		\frac{\partial^\alpha  u}{\partial t^\alpha}(t)&= \frac{\partial u}{\partial t}, \quad \textrm{ for } \quad \alpha =1, 
	\end{align}
	where $\Gamma(.)$ is the Gamma function. Problem \eqref{problem1}-\eqref{problem4} is an  ill-posed inverse problem (see Lemma 2.1);  the solution does not depend continuously on the given data, i.e.,  any small perturbation in the given data may cause a large change to the solution. The sideway{\color{black}s} problem with classical derivative has been considered by many authors.   In 1995, Teresa Regiflska \cite{Regi}   solved a   sideways heat problem which consists in applying {\color{black} the} wavelet basis decomposition of measured data in the quarter plane ($x\ge 0$, $t\ge 0$).  In 1999, F. Berntsson \cite{Bern}   investigated  the following sideways heat equation	
	\begin{align}
		ku_{xx} & =u_t, \qquad 0<x <1,\quad t\ge 0, \nn\\
		u(1,t)&=g(t), \hspace*{0.5cm} t \ge 0, \nn\\
		u_x(1,t)&=h(t), \hspace*{0.5 cm}  t \ge 0, \nn\\
		u(x,0)&=0, \hspace*{0.95cm} 0<x<1, \nn
	\end{align}	
	where $g$, $h$ were the functions be defined later. He tried to use the spectral method to determine the temperature $u(x,t)$ for $0\le x<1$ from temperature measurements $g=u(1,.)$ and heat-flux measurements $h=u_x(1,.)$. Later on, in 2010, T. Wei  \cite{Wei1}  proposed a spectral regularization method for  the following time fractional advection-dispersion equation 	
	\begin{align}
		{}_0D^\alpha_t u + bu_x =au_{xx} & , \qquad 0<x <L,\quad t> 0, \nn\\
		u(0,t)&=f(t), \hspace*{0.5cm} t \ge 0, \nn\\
		u_x(0,t)&=g(t), \hspace*{0.5 cm}  t \ge 0, \nn\\
		u(x,0)&=0, \hspace*{0.95cm} 0<x<L. \nn
	\end{align}
	where $f$, $g$ were the functions be defined later.

	Recently, the homogeneous problem, i.e, $f \equiv 0$ in Eq. \eqref{problem1} has been considered by some authors, for example{\color{black},} see \cite{Wei1,  Wei4, Wei5, Xiong2}. Although there are many papers on the homogeneous case of the identification problem, the inhomogeneous case has not been intensively investigated. The inhomogeneous case is first studied by G.H. Zheng \cite{zheng}. Very recent, Tuan and his group \cite{Tuan} have studied {\color{black}a} more general case of inhomogeneous source term in the form $f(x,t)$. 	Until now, to our knowledge, the problem  \eqref{problem1}-\eqref{problem4} with a generalized source term  $f(x,t,u)$  has not been studied. Obviously, the nonlinear
problem is much more challenging.   In case of the homogeneous problem, we can transform the solution $u$ into the linear equation
	\begin{equation}
	u(1,t)= \mathbb{T} (g,h){\color{black},} \label{linearoperator}
	\end{equation}
    where $\mathbb{T}$ is {\color{black}a} linear unbounded operator. {\color{black} Then, there are} many choices of stability term for regularization that have been proposed. The main idea is to replace the operator $ \mathbb{T} $ with a class of linear bounded ope{\color{black}ra}tor.  However,  when the right {\color{black} handside } of \eqref{problem1} depend on $u$,  it is impossible to represent $u$ as \eqref{linearoperator}. Thus, the techniques and methods in previous papers on the homogeneous case cannot be applied directly to solve the {\color{black} nonlinear } inhomogeneous problem.  Now, we describe our new ideas for the nonlinear inhomogeneous problem. The sought solution of  Problem  \eqref{problem1}-\eqref{problem4}  can be represented by a nonlinear integral equation containing some instability terms, see  
    \eqref{exsolution}. Our main objectives are    to find a suitable integral equation for approximating the
    exact solution by replacing the instability terms with regularization terms and then  show that the solution of regularized problem
    converges to the exact solution. Notice that the following sideway{\color{black}s} problem
for time fractional diffusion equation 
	\begin{equation} \label{problemm}
	\begin{gathered}
	-  u_x(x,t)= D^\alpha_tu(x,t) + f(x,t,u(x,t)), \quad x > 0,\; t>0,\\
	u(1,t)=g(t), \quad t \ge 0, \\
	\lim_{x\to +\infty} u(x,t) =u(x,0)=0, \quad t \ge 0,
	\end{gathered}
	\end{equation}
	has been studied by M. Kirane et al \cite{Tuan2}. Our problem \eqref{problem1}-\eqref{problem4} is more complicated than   \eqref{problemm} since there are two {\color{black} boundary } functions in \eqref{problem3}-\eqref{problem4} to be investigated. Moreover, in this paper, we first give the convergence rate in  $H^p$ norm  which  is not considered in  \cite{Tuan2} and some previous {\color{black} papers} \cite{Tuan,Wei1,  Wei4, Wei5,  Xiong2} .

The outline of this work is as follows. In Section 2, we present the ill-posedness
of the problem.   In Section 3, we  propose our  regularization method and  convergence estimates for the regularized solution and the sought solution   are given in both $L^2$- and $H^p$-norm based on the {\color{black} a} priori assumptions.  Finally, in
Section 4 we implement a numerical example to illustrate the theoretical results.

	\section{Ill-posedness of the nonlinear problem }
	
	\subsection{The mild solution of  Problem \eqref{problem1}-\eqref{problem4} }
	
	\noindent To apply the Fourier transform, thanks to \cite{Chorfi}, we extend all functions in this paper to the
	whole line $-\infty < t < +\infty$ by defining them to be zero for $t<0$. 
	The Fourier transform of $L^2(\mathbb{R})$ function $v(t)$ $(-\infty < t <\infty)$ is defined by 
	$$\displaystyle \mathcal{F} (v)(\omega):= \widehat{v}(\omega):=\frac{1}{\sqrt{2\pi}}\int_{-\infty}^{\infty} v(t)e^{-i\omega t}dt, \quad \omega \in \mathbb{R} .$$ 
	We denote by $\norm{.}_{L^2(\mathbb{R})}$   the $L^2(\mathbb{R})$ norm, i.e.,
	$$\displaystyle\norm{v}_{L^2(\mathbb{R})} =\left( \int_{\mathbb{R}} |v(\omega)|^2d\omega \right)^{\frac{1}{2}},$$ and by $\displaystyle\norm{.}_{H^p(\mathbb{R})}$   the $H^p(\mathbb{R})$ norm, i.e.,
$$\displaystyle\norm{v}_{H^p(\mathbb{R})}=\left( \int_{\mathbb{R}} (1+\omega^2)^p|\widehat v(\omega)|^2d\omega \right)^{\frac{1}{2}}.$$
Note that when $p=0$, $H^p(\mathbb{R})= H^0(\mathbb{R})=L^2(\mathbb{R})  $. \\
 
	Applying Fourier transform with respect to variable $t$ to the problem (\ref{problem1}) - (\ref{problem4}), we obtain  the following second order differential equation
	\begin{equation}
		\begin{cases}
			\begin{array}{llll}
				(i\omega)^\alpha \, \widehat{u}(x,\omega) - \hat{u}_{xx}(x,\omega)&=\,\,\widehat{f}(x,\omega,u(x,\omega)), & & x > 0, \omega \in \mathbb{R},\\
				\widehat{u}(0,\omega)&= \,\, \widehat{g}(\omega), & & \omega\in\mathbb{R},\\
				\widehat{u}_x(0,\omega)&= \,\, \widehat{h}(\omega), & & \omega\in\mathbb{R}.\\
			\end{array}
		\end{cases}\label{1111}
	\end{equation}
	
Put 
	\begin{align}k(\omega):=(i\omega)^{\frac{\alpha}{2}}= \Re(k(\omega)) +i\Im(k(\omega)) ,
	\end{align}
where the real   and  image parts of $ k(\omega)$  are respectively
	\begin{align}
		\Re(k(\omega))  := |\omega|^{\frac{\alpha}{2}} \cos{\frac{\alpha \pi}{4}},\quad \Im(k(\omega))= |\omega|^{\frac{\alpha}{2}}\textrm{sign}(\omega)\sin{\frac{\alpha \pi}{4}} . \label{rexi} 
	\end{align} 	
	Multiplying the first equation of  \eqref{1111} by $\displaystyle \frac{ \sinh {\Big (k(\omega)(x-z)\Big) }  }{k(\omega)}$    and integrating two sides on $[0;x]$, we derive
	\begin{align}
	\int_0^x \Big(  	(i\omega)^\alpha \, \widehat{u}(z,\omega) - \hat{u}_{zz}(z,\omega) \Big) \frac{ \sinh {\Big (k(\omega)(x-z)\Big) }  }{k(\omega)}dz   = \int_0^x \frac{ \sinh {\Big (k(\omega)(x-z)\Big) }  }{k(\omega)} \widehat{f}(z,\omega,u(z,\omega))  dz. \label{e222}
	\end{align}
	By applying integration by parts to the left   side of \eqref{e222}, and combining the second and third equation of \eqref{1111}, we obtain
	\begin{align}
		 \widehat{u}(x,\omega)&= \cosh \Big(k(\omega) x\Big) \widehat{g}(\omega) +  \frac{\sinh \big(k(\omega) x\Big)}{k(\omega)} \widehat{h}(\omega) \nn\\
		& -    \int_0^x \frac{ \sinh {\Big (k(\omega)(x-z)\Big) }  }{k(\omega)} \widehat{f}(z,\omega,u(z,\omega))  dz,  \label{solution3} 
	\end{align}
	for $x\ge 0$, $\omega\in\mathbb{R}$. Applying the inverse Fourier transform to (\ref{solution3}), we have
	\begin{align} 
		u(x,t)=& \frac{1}{\sqrt{2\pi}}\int_{-\infty}^{+\infty}{ 
			\cosh \Big(k(\omega) x\Big) \hat{g}(\omega)  e^{i\omega t}d\omega} +
		\frac{1}{\sqrt{2\pi}}\int_{-\infty}^{+\infty}{ \frac{\sinh \Big(k(\omega) x\Big)}{k(\omega)} \widehat{h}(\omega) e^{i\omega t}d\omega} \nn \\
		&-    \frac{1}{\sqrt{2\pi}}\int_{-\infty}^{+\infty}\int_0^x  \frac{ \sinh {\Big (k(\omega)(x-z)\Big) }  }{k(\omega)} \widehat{f}(z,\omega,u(z,\omega))e^{i\omega t}dzd\omega   , \label{exsolution}
	\end{align}
	for $x\ge 0$, $t\in\mathbb{R}$. \\
	Note that the  real part of $ k(\omega)$  is an increasing positive function of $\omega$.  Therefore the terms $$\cosh\Big((k(\omega) x\Big), \frac{\sinh\Big((k(\omega) x\Big)}{|k(\omega)|}$$ 
	increase  rather quickly when $|\omega| \to \infty$
	: small errors in high-frequency components
	can blow up and completely destroy the solution for $0<x<1$ .  Therefore, the problem is severely ill-posed {\color{black}and} regularization methods are required for finding the approximate solution of  our problem. 
	
		\subsection{Some notations}
		We prove the following lemma which will be important to obtain {\color{black} the} main results.
		\begin{lemma}\label{ineqlemma} 
			Let $\Re(z)$ be  the real part of any complex number $z$.  If $\Re(z) >0$ then we have
			\begin{itemize}
				\item[a)] $\displaystyle |\cosh (z)|  \le e^{\Re{(z)}}$; \vspace*{0.2cm}
				\item[b)] $\displaystyle \left|\frac{\sinh (\lambda z)}{z}\right| \le \lambda e^{\lambda \Re(z)}$.
			\end{itemize}	
		\end{lemma} 
		
		\noindent {\bf Proof.} a) We have 
\begin{align}
|\cosh z| &=\frac{\left|e^{z}+e^{-z}\right|}{2} = \frac{\left|e^{\Re{(z)} +i\Im(z)}+e^{- \Re{(z)} -i\Im(z)}\right|}{2}  \nonumber\\
& = \frac{\left|e^{\Re{(z)}}\left( \cos \Im(z) + i \sin \Im(z) \right) + e^{-\Re{(z)}}\left( \cos \Im(z) - i \sin \Im(z) \right) \right|}{2}   \nonumber\\
& =\frac{ \left| \left(e^{\Re{(z)}}+e^{-\Re{(z)}}\right)  \cos \Im(z) + i \sin \Im(z) \left( e^{\Re{(z)}}  - e^{-\Re{(z)}}  \right) \right|}{2}  \nonumber\\
& = \frac{\sqrt{\left(e^{\Re{(z)}}+e^{-\Re{(z)}}\right)^2  \cos^2 \Im(z) + \left(e^{\Re{(z)}}-e^{-\Re{(z)}}\right)^2  \sin^2 \Im(z) }}{2} \nonumber\\
& = \frac{\sqrt{ e^{2\Re{(z)}} +2\cos{2\Im(z)} + e^{-2\Re{(z)}}  }}{2} \nonumber\\
& \le \frac{ e^{\Re{(z)}}  + e^{-\Re{(z)}} }{2} =\cosh \Re(z) \le e^{\Re{(z)}} ~ \textrm{for} ~ \Re(z)> 0.   \nonumber
\end{align} 		

\vspace*{0.3cm}
		
	\noindent	b) First, we have 
		\begin{align}
		\left|\frac{\sinh (z\lambda) }{z} \right|& = \left| \int_0^\lambda  \cosh(s z) ds \right|   \le  \int_0^\lambda |\cosh(sz )| ds    \le  \int_0^\lambda e^{\Re(z)s } ds    = \frac{e^{\lambda\Re(z) }-1} {\Re(z)}. \label{lemma1prove1}
		\end{align}  
		Second, we have 
		\begin{align}
	 \frac{e^{\lambda\Re(z) }-1} {\Re(z)} = \,& \frac{\sum\limits_{k=0}^{+\infty}\frac{\left( \lambda\Re(z) \right)^k}{k!} - 1 }{ \Re(z) } = \frac{\sum\limits_{k=1}^{+\infty}\frac{\left( \lambda\Re(z) \right)^k}{k!}  }{ \Re(z) } \nn\\
		= & \lambda \sum\limits_{k=1}^{+\infty}\frac{\left( \lambda\Re(z) \right)^{k-1} }{k!}= \lambda \sum\limits_{l=0}^{+\infty}\frac{\left( \lambda\Re(z) \right)^{l} }{(l+1)!}  \nn\\  
		\le \,&  \lambda \sum\limits_{l=0}^{+\infty}\frac{\left( \lambda\Re(z) \right)^{l} }{l!} \le  \lambda e^{\lambda \Re(z) }. \label{lemma1prove2}
		\end{align}
		
		The inequality in part b) follows from    (\ref{lemma1prove1}) and (\ref{lemma1prove2}).$\square$\\

\begin{lemma}\label{ineqlemma2} Let  $\epsilon$, $\xi$, $p$, $\gamma$ be  positive real numbers.   If $\ep$ satisfies the following condition
		\begin{align}
		\epsilon < \left[\frac{\xi \gamma\cos\frac{\alpha \pi}{4}}{p}\right]^{ \frac{1}{\xi}}, \label{lemmamonoas2}
		\end{align}   then we have
		\begin{align}
		(1+\omega^2)^p \exp\left(  2(x-1-\gamma)\omega^{\xi} \cos \frac{\alpha \pi}{4} \right) \le \left(1+\epsilon^{-2}\right)^p \exp\left(  2(x-1-\gamma)\ep^{-\xi} \cos \frac{\alpha \pi}{4} \right) \label{lemmamonoas3}
		\end{align} for all $0\le x <  1$ and $\omega \ge \displaystyle \frac{1}{\epsilon}$.
	\end{lemma}
	
	\noindent {\bf Proof.} For $0\le x <1$, we define  
$$	\Lambda(\omega):= (1+\omega^2)^p \exp\left(  2(x-1-\gamma)\omega^{\xi} \cos \frac{\alpha \pi}{4} \right),$$ for $0\le x \le  1$.  Let us denote
$$\Gamma(\omega)=2\omega p (1+\omega^2)^{p-1}\exp\left(  2(x-1-\gamma)\omega^{\xi} \cos \frac{\alpha \pi}{4} \right) $$
then
		\begin{align}
		& \frac{d\Lambda}{d \omega} =  \Gamma(\omega) \left[ 1 - \frac{ (1+\omega^2)(1-x+\gamma)   \cos \frac{\alpha \pi}{4} .\xi.\omega^{\xi-2}}{ p}  \right] \nn\\
		   \le\,&  \Gamma(\omega) \left[ 1 - \frac{  \omega^2   \gamma    \cos \frac{\alpha \pi}{4} .\xi.\omega^{\xi-2}}{ p}  \right] \nn\\
		 \le\,&  \Gamma(\omega) \left[ 1 - \frac{ \xi \gamma \cos \frac{\alpha \pi}{4} . \omega^{\xi }}{ p}  \right]  \le  \Gamma(\omega) \left[ 1 - \frac{ \xi \gamma \cos \frac{\alpha \pi}{4} . \ep^{-\xi }}{ p}  \right]  \label{lemmamonoas4} 
		\end{align} since $\omega \ge \displaystyle \frac{1}{\epsilon}$. 
It follows from (\ref{lemmamonoas2}) that $\displaystyle 1 - \frac{ \xi \gamma \cos \frac{\alpha \pi}{4} . \ep^{-\xi }}{ p} < 0$. Therefore, the inequality (\ref{lemmamonoas4}) implies $\displaystyle \frac{d \Lambda}{d \omega} \le  0$ for $\omega \ge \displaystyle \frac{1}{\epsilon}$. The proof is completed.$\Box$

\subsection{An example of ill-posedness for problem \eqref{problem1}-\eqref{problem4}  } In this subsection, we give  an example of ill-posedness  by choosing the function $f$ as follows 
\begin{align}
f(z,t,u(z,t))=\frac{1}{\sqrt{2\pi}}\int_{-\infty}^{+\infty} \frac{1}{2} e^{-k(\omega)}  \widehat u(z,\omega) e^{i\omega t } d\omega
  \label{exf}  
\end{align} or
\begin{align}
\widehat{f}(z,\omega,u(z,\omega))=\frac{1}{2} e^{-k(\omega)}  \widehat u(z,\omega)
, \label{exfb} 
\end{align} or 
 for all $(z,\omega)\in [0;1]\times[0;+\infty)$ and it is extended to  zero for all $(z,\omega)\in [0;1]\times(-\infty;0)$.   The ill-posedness of the problem \eqref{problem1}-\eqref{problem4} corresponding to the above function $f$ can be proved by using the following lemmas. 
 
\begin{lemma}
Let $f$ be defined as \eqref{exf}. Then for any $(g,h) \in \Big( L^2(\mathbb{R})\Big)^2 $,  Problem  \eqref{exsolution} has unique solution $u^*(g,h) \in C([0,1];L^2(\mathbb{R})) $.
\end{lemma}

\begin{proof}
Let us set
\begin{align} 
		\Phi(v)(x,t):=& \frac{1}{\sqrt{2\pi}}\int_{-\infty}^{+\infty}{ 
			\cosh \Big(k(\omega) x\Big) \hat{g}(\omega)  e^{i\omega t}d\omega} +
		\frac{1}{\sqrt{2\pi}}\int_{-\infty}^{+\infty}{ \frac{\sinh \Big(k(\omega) x\Big)}{k(\omega)} \widehat{h}(\omega) e^{i\omega t}d\omega} \nn \\
		&-    \frac{1}{\sqrt{2\pi}}\int_{-\infty}^{+\infty}\int_0^x  \frac{ \sinh {\Big (k(\omega)(x-z)\Big) }  }{k(\omega)} \widehat{f}(z,\omega,v(z,\omega))e^{i\omega t}dzd\omega   \nn
	\end{align} or
\begin{align}
\Phi(v)(x,t):=& \frac{1}{\sqrt{2\pi}}\int_{-\infty}^{+\infty}{ 
	\cosh \Big(k(\omega) x\Big) \hat{g}(\omega)  e^{i\omega t}d\omega} +
\frac{1}{\sqrt{2\pi}}\int_{-\infty}^{+\infty}{ \frac{\sinh \Big(k(\omega) x\Big)}{k(\omega)} \widehat{h}(\omega) e^{i\omega t}d\omega} \nn \\
&-   \frac{1}{2\sqrt{2\pi}}\int_{-\infty}^{+\infty} \int_0^x \frac{ \sinh {\Big (k(\omega)(x-z)\Big) }  }{k(\omega)} e^{-k(\omega)}  \widehat v(z,\omega) e^{i\omega t} dz d\omega,\nn 
\end{align}
for all $ v\in C([0,1];L^2(\mathbb{R}))$. For any  $(g,h) \in L^2(\mathbb{R}) \times L^2(\mathbb{R})  $, we  claim that 
the following equation 
\begin{align}
	v(x,t)=\Phi(v)(x,t)  \label{example1}  
	\end{align} has  a unique solution  in $C([0,1];L^2(\mathbb{R}))$.  
 For $v_1, v_2\in C([0,1];L^2(\mathbb{R}))$ and $0\le x\le1$, we shall prove that
 \begin{align}
 \|\Phi(v_1)(x,.)-\Phi(v_2)(x,.)\|^2_{L^2(\mathbb{R})}  \le    \frac{1}{2}\,\, \|v_1-v_2\|^2 _{C([0,1];L^2(\mathbb{R}))}. \label{exexqn}
 \end{align}     Indeed, we have  		
 \begin{align}
 & \Phi(v_1)(x,t)-\Phi(v_2)(x,t) \nn\\
      = \,&   - \frac{1}{2\sqrt{2\pi}}\int_{-\infty}^{+\infty}\int_0^x  \frac{ \sinh {\Big (k(\omega)(x-z)\Big) }  }{k(\omega)} e^{-k(\omega)}\left[\widehat v_1(z,\omega ) - \widehat v_2(z,\omega ) \right] dz e^{i\omega t}d\omega \nn\\
  =\,&  \mathcal{F}^{-1}\left( -\frac{1}{2}\int_0^x  \frac{ \sinh {\Big (k(\omega)(x-z)\Big) }  }{k(\omega)} e^{-k(\omega)}\left[\widehat v_1(z,\omega ) - \widehat v_2(z,\omega ) \right]dz\right)(t). \nn
 \end{align} 
By using the Parseval's identity, we obtain 
 \begin{align}
&   \|\Phi(v_1)(x,.)-\Phi(v_2)(x,.)\|_{L^2(\mathbb{R})}^2  \nn\\
  =\,&  \int_{-\infty}^{+\infty} \left| \Phi(v_1)(x,t)-\Phi(v_2)(x,t) \right|^2 dt  \nn\\
 = \,&  \frac{1}{2}\int_{-\infty}^{+\infty}   \left| \int_0^x  \frac{ \sinh {\Big (k(\omega)(x-z)\Big) }  }{k(\omega)} e^{-k(\omega)}\left[\widehat v_1(z,\omega ) - \widehat v_2(z,\omega ) \right]dz \right|^2   d\omega. \label{exex3}
 \end{align}
 Applying the H\"older inequality to (\ref{exex3}), we derive
 \begin{align}
 &   \left| \int_0^x  \frac{ \sinh {\Big (k(\omega)(x-z)\Big) }  }{k(\omega)} e^{-k(\omega)}\left[\widehat v_1(z,\omega ) - \widehat v_2(z,\omega ) \right]dz\right|^2 \nn\\
  \le\,&    x   \int_0^x \left| \frac{ \sinh {\Big (k(\omega)(x-z)\Big) }  }{k(\omega)}\right|^2  
 \left|e^{-k(\omega)}\left[\widehat v_1(z,\omega ) - \widehat v_2(z,\omega ) \right]\right|^2 dz   \nn\\
 \le\,&     \int_0^x (x-z)^2 \exp\Big(2\Re(k(\omega))(x-z)\Big)  \exp\Big(-2\Re(k(\omega)) \Big) 
 \left|\widehat v_1(z,\omega ) - \widehat v_2(z,\omega )\right|^2 dz     \nn\\
  \le \,&     \int_0^x  
 \left|\widehat v_1(z,\omega ) - \widehat v_2(z,\omega )\right|^2 dz,     \label{exex4}
 \end{align} where we have used the Lemma \ref{ineqlemma} as follows
\begin{align}
\left| \frac{ \sinh {\Big (k(\omega)(x-z)\Big) }  }{k(\omega)}\right|^2 \le (x-z)^2 \exp\Big(2\Re(k(\omega))(x-z)\Big). \label{appliedlemma} 
\end{align}

 It follows from (\ref{exex3}), and (\ref{exex4}) that 
 \begin{align} \|\Phi(v_1)(x,.)-\Phi(v_2)(x,.)\|_{L^2(\mathbb{R})}^2 
  \le \,\,& \frac{1}{2} \int_{-\infty}^{+\infty}     \int_0^x   \left|\widehat v_1(z,\omega ) - \widehat v_2(z,\omega )\right|^2 dz d\omega \nn\\
 \le \,\,& \frac{1}{2}\int_0^x         \norm{\widehat v_1(z,. ) - \widehat v_2(z,. )}^2_{L^2(\mathbb{R})}  dz \nn\\
  \le \,\,& \frac{1}{2} \int_0^x         \norm{  v_1(z,. ) -   v_2(z,. )}^2_{L^2(\mathbb{R})}  dz  \label{exex5}\\
 \le \,\,& \,\, \frac{1}{2}\|v_1-v_2\|^2 _{C([0,1];L^2(\mathbb{R}))} . \label{exex6}
 \end{align}	
Consequently, we get
 \begin{align}
 \|\Phi(v_1) -\Phi(v_2) \|^2_{C([0,1];L^2(\mathbb{R}))}  \le    \frac{1}{2}\,\, \|v_1-v_2\|^2 _{C([0,1];L^2(\mathbb{R}))}, 
 \end{align}  so $\Phi$ is a contract mapping on $C([0,1];L^2(\mathbb{R}))$.  Thus, there is a unique fixed point of $\displaystyle\Phi $ in $C([0,1];L^2(\mathbb{R}))$ which is denoted by $u^*(g,h)$, i.e. $\Phi(u^*(g,h))=u^*(g,h)$. We obtain the existence and uniqueness of the solution of \eqref{example1}. Hence, the problem  \eqref{exsolution} has a unique solution $u^*(g,h) \in C([0,1];L^2(\mathbb{R})) $.
\end{proof}

\begin{lemma}
Let $f$ be defined as \eqref{exf}. Then for any $(g,h) \in \Big( L^2(\mathbb{R})\Big)^2 $,  Problem  \eqref{exsolution} is unstable.
\end{lemma}

\begin{proof}

  To show the instability of $u$, we construct the functions $g_0=h_0=0$ and $ (g_n, h_n)$
 defined by the Fourier transform, as follows:
 $$\widehat{g_n}(\omega)= \displaystyle \frac{1}{k(\omega)} \chi_{\left[n;\,n+\frac{1}{n}\right]}(\omega),~~\widehat{h_n}(\omega)=\displaystyle\chi_{\left[n;\,n+\frac{1}{n}\right]}(\omega),$$ for all $\omega \in \mathbb{R}$. It is easy to check that 
 \[
  \norm{g_n-g_0}_{L^2(\mathbb{R})}+ \norm{h_n-h_0}_{L^2(\mathbb{R})} \to 0,
 \]
 when $n \to 0$. Indeed,  we have 
 \begin{align}
  \norm{g_n-g_0}_{L^2(\mathbb{R})}^2=\norm{ g_n}_{L^2(\mathbb{R})}^2=&\,\,\norm{\widehat g_n}_{L^2(\mathbb{R})}^2 
  =\int_{n}^{\, n+\frac{1}{n}} \left| \frac{1}{k(\omega)} \right|^2 d\omega. \nn
\end{align}  
We note that $\displaystyle  \omega^\alpha > 1$  for all $\displaystyle \omega\in \left[n;\,n+\frac{1}{n}\right] $. So 
\begin{align}
\int_{n}^{\, n+\frac{1}{n}} \left| \frac{1}{k(\omega)} \right|^2 d\omega = \int_{n}^{\, n+\frac{1}{n}}  \frac{1}{ \omega^{\alpha} }  d\omega < \int_{n}^{\, n+\frac{1}{n}}    d\omega < \frac{1}{n} \longrightarrow 0 \quad \textrm{when } n \to +\infty .\nn
\end{align}
This implies $\displaystyle \norm{g_n-g_0}_{L^2(\mathbb{R})}^2 \longrightarrow 0 $ when $n \to +\infty$. Moreover, 
  \begin{align}
  \norm{h_n-h_0}_{L^2(\mathbb{R})}^2=\norm{h_n}_{L^2(\mathbb{R})}^2 =&\norm{\widehat h_n }_{L^2(\mathbb{R})}^2=\int_{n}^{\, n+\frac{1}{n}}  d\omega =\frac{1}{n} \longrightarrow 0 \quad \textrm{when } n \to +\infty.
  \end{align}

 Let $u_n(g_n,h_n)$ and $u(g_0, h_0)$ be two solutions of  Problem  \eqref{exsolution} corresponding to $(g_n,h_n)$ and $(g_0, h_0)$ respectively. The existence of $u_n(g_n,h_n)$ and $u(g_0, h_0)$ has been proved in Lemma 2.3.   We will show that 
 $\|u_n(g_n,h_n)-u(g_0, h_0) \|_{L^2(\mathbb{R})}$ does not converge to zero when $n \to +\infty$. 
 Since
 \begin{align}
 \widehat{u_n}(g_n,h_n)(x,\omega)  
 = \,& \cosh \Big(k(\omega) x\Big)  \frac{1}{k(\omega)}\chi_{\left[n;\,n+\frac{1}{n}\right]} +  \frac{\sinh \big(k(\omega) x\Big)}{k(\omega)}  \chi_{\left[n;\,n+\frac{1}{n}\right]}(\omega) \nn\\
 & \hspace*{3.1cm} -   \frac{1}{2} \int_0^x \frac{ \sinh {\Big (k(\omega)(x-z)\Big) }  }{k(\omega)} e^{-k(\omega)}  \widehat u_n(g_n,h_n)(z,\omega)  dz  \nonumber \nn \\
   = \,& \frac{e^{k(\omega) x}}{k(\omega)}.\chi_{\left[n;\,n+\frac{1}{n}\right]}(\omega)  -  \frac{1}{2}  \int_0^x \frac{ \sinh {\Big (k(\omega)(x-z)\Big) }  }{k(\omega)} e^{-k(\omega)}  \widehat u_n(g_n,h_n)(z,\omega)  dz.  \label{2.25}
 \end{align}
We  get
\begin{align}
 \widehat{u_n}(g_n,h_n)(x,\omega)-\widehat{u}(g_0,h_0)(x,\omega) = A(x,\omega) - B(x,\omega) , \label{exa1} 
\end{align}
where   $$A(x,\omega):=\frac{e^{k(\omega) x}}{k(\omega)}\chi_{\left[n;\,n+\frac{1}{n}\right]}(\omega)$$ 
and 
$$B(x,\omega):= \frac{1}{2} \int_0^x \frac{ \sinh {\Big (k(\omega)(x-z)\Big) }  }{k(\omega)} e^{-k(\omega)}\Big[  \widehat u_n(g_n,h_n)(z,\omega) -  \widehat u(g_0,h_0)(z,\omega) \Big] dz.$$
If $n$ is  large enough, then
\begin{align}
\Big|A(x,\omega)\Big| = \,\, \frac{e^{\,\mathcal{R}(k(\omega))x}}{\omega^{\,\,{\frac{\alpha}{2}}}} = \,\, \frac{\exp\left(x\cos\frac{\alpha\pi}{4}.\omega^{\,\,{\frac{\alpha}{2}}}\right) }{\omega^{\,\,{\frac{\alpha}{2}}}}  \ge   \sqrt{6}. \omega, \label{exa2} 
\end{align} for $\displaystyle \omega\in\left[n;\,n+\frac{1}{n}\right]$. On the other hand, we have 
\begin{align*}
\Big|B(x,\omega)\Big|^2 \le  \frac{1}{2} \int_0^x \left|  \widehat u_n(g_n,h_n)(z,\omega) -  \widehat u(g_0,h_0)(z,\omega) \right|^2 dz
\end{align*} by using the same way in (\ref{exex4}). Hence,
\begin{align}
& \int_{\left[n;\,n+\frac{1}{n}\right]} \left|B(x,\omega)\right|^2d\omega\le   \int_{-\infty}^{+\infty} \left|B(x,\omega)\right|^2d\omega \nn\\
  \le \,\,& \frac{1}{2} \int_0^x \int_{-\infty}^{+\infty} \left|  \widehat u_n(g_n,h_n)(z,\omega) -  \widehat u(g_0,h_0)(z,\omega) \right|^2  d\omega dz \nn\\
 \le \,\,& \frac{1}{2}\int_0^x \norm{\widehat u_n(g_n,h_n)(z,.)-\widehat u(g_0,h_0)(z,.)}_{L^2(\mathbb{R})}^2 dz   \nn\\
 \le \,\,& \frac{1}{2}\int_0^x \norm{  u_n(g_n,h_n)(z,.)-  u(g_0,h_0)(z,.)}_{L^2(\mathbb{R})}^2 dz. \label{exa3}
\end{align} 
It follows from (\ref{exa1}), (\ref{exa2}), and (\ref{exa3}) that 
\begin{align}
&\norm{u_n(g_n,h_n)(x,.)-u(g_0,h_0)(x,.)}_{L^2(\mathbb{R})}^2 \nn\\
= &\,\, \int_{-\infty}^{+\infty} \left| A(x,\omega)-B(x,\omega) \right|^2d\omega \nn\\
\ge &\,\, \int_{\left[n;\,n+\frac{1}{n}\right]} \left( \frac{\left| A(x,\omega)\right|^2}{2}- \left|B(x,\omega) \right|^2\right)d\omega\nn\\
\ge &\,\, \int_{\left[n;\,n+\frac{1}{n}\right]} 3\omega^2 d\omega -   \int_{\left[n;\,n+\frac{1}{n}\right]} \left|B(x,\omega) \right|^2 d\omega\nn\\
\ge &\,\, \int_{\left[n;\,n+\frac{1}{n}\right]} 3\omega^2 d\omega - \frac{1}{2} \int_0^x \norm{u_n(g_n,h_n)(z,.)-u(g_0,h_0)(z,.)}_{L^2(\mathbb{R})}^2 dz \nn\\
\ge &\,\, \int_{\left[n;\,n+\frac{1}{n}\right]} 3\omega^2 d\omega - \frac{1}{2} \sup_{0 \le z \le 1} \norm{u_n(g_n,h_n)(z,.)-u(g_0,h_0)(z,.)}_{L^2(\mathbb{R})}^2, \label{exa4}
\end{align} where we have used  the following inequalities $$\left| z_1-z_2 \right|^2 \ge \Big| \left| z_1\right|-\left|z_2 \right| \Big|^2 \ge \frac{\left| z_1\right|^2}{2}- \left|z_2 \right|^2,$$ for all complex numbers $z_1$ and $z_2$. From (\ref{exa4}) we have
\begin{align}
&\,\,\norm{u_n(g_n,h_n)(x,.)-u(g_0,h_0)(x,.)}_{L^2(\mathbb{R})}^2  + \frac{1}{2} \sup_{0 \le z \le 1} \norm{u_n(g_n,h_n)(z,.)-u(g_0,h_0)(z,.)}_{L^2(\mathbb{R})}^2 \ge n \label{exa5}
\end{align} since $$\int_{\left[n;\,n+\frac{1}{n}\right]} 3\omega^2d\omega =   \left[\left(n+\frac{1}{n}\right)^3-n^3\right] \ge n.$$ 
The left side of (\ref{exa5}) is less than $\displaystyle \frac{3}{2} \sup_{0 \le x \le 1} \norm{u_n(g_n,h_n)(x,.)-u(g_0,h_0)(x,.)}_{L^2(\mathbb{R})}^2$. Hence, we obtain
\begin{equation}
\sup_{0 \le x \le 1} \norm{u_n(g_n,h_n)(x,.)-u(g_0,h_0)(x,.)}_{L^2(\mathbb{R})}^2 \ge \frac{2}{3}n.
\end{equation}
The above inequality implies that Problem  \eqref{exsolution} is  ill-posed in the Hadamard sense in $L^2$-norm. 
\end{proof}

{\color{black} 
\begin{remark}
	Notice that the solution of  Problem \eqref{problem1}-\eqref{problem4} in the case $F=0$ is 
	\begin{align} 
	u(x,t)= \frac{1}{\sqrt{2\pi}}\int_{-\infty}^{+\infty}{ 
		\cosh \Big(k(\omega) x\Big) \hat{g}(\omega)  e^{i\omega t}d\omega}  \label{exsolution22}
	\end{align}	
	The growth factor $ |	\cosh \Big(k(\omega) x\Big)| \simeq \exp \Big(|\omega|^{\frac{\alpha}{2}} \cos{\frac{\alpha \pi}{4}} \Big) $ leads to the severely ill-posed nature of the homogeneous problem.  Then, the degree of ill-posedness  of our problem is  $\exp \Big(|\omega|^{\frac{\alpha}{2}} \cos{\frac{\alpha \pi}{4}} \Big)$.  As in \cite{Fre}, we can see that the degree of ill-posedness for the
	classical parabolic problem is $\exp \Big(|\omega|^{\frac{1}{2}}  \Big)$. 
	Since $0<\alpha <1$, we have that $\exp \Big(|\omega|^{\frac{\alpha}{2}} \cos{\frac{\alpha \pi}{4}} \Big)$  grows at a slower rate than $\exp \Big(|\omega|^{\frac{1}{2}}  \Big)$ does. This implies that the  fractional  case of  parabolic 
	problem is ``less ill-posed'' than the classical  one.
\end{remark}
}

\section{Regularization and error estimate for  problem \eqref{exsolution} }
\subsection{Regularized solution}
\noindent In order to obtain a stable approximate solution of the problem, we  apply the truncation method. Let $\displaystyle (g^{\delta},h^{\delta})\in \left( L^2(\mathbb{R}) \right)^2$  be the measured data which satisfy
\begin{align}
\norm{g^{\delta}-g }_{L^2(\mathbb{R})} + \norm{h^{\delta}-h }_{L^2(\mathbb{R})}  \le \delta, \label{noisyassumption}
\end{align}
where the constant $\delta>0 $ is called the error level. We present the following regularization problem  
\begin{align} 
\widehat{w}(x,\omega)&= \cosh \Big(k(\omega) x\Big) \hat{g}_\epsilon^\delta(\omega) + \frac{\sinh \Big(k(\omega) x\Big)}{k(\omega)} \widehat{h}_\ep^\delta(\omega)   -    \int_0^x  \frac{ \sinh {\Big (k(\omega)(x-z)\Big) }  }{k(\omega)} \widehat{f}_\epsilon(z,\omega,w(z,\omega))dz   ,  \label{regularizedb1}
\end{align}
for $x\ge 0$,   $\omega\in\mathbb{R}$, where $\epsilon:=\epsilon(\delta)>0$ is a regularization parameter and 
$$\Big[\widehat{g}^\delta_\epsilon(\omega) \quad \widehat{h}^\delta_\epsilon(\omega) \quad \widehat{f}_\epsilon(z,\omega,w(z,\omega)) \Big] := \Big[\widehat{g}^\delta(\omega) \quad \widehat{h}^\delta(\omega) \quad \widehat{f}(z,\omega,w(z,\omega)) \Big]\chi_{\left[-\frac{1}{\epsilon};\frac{1}{\epsilon}  \right]}(\omega).$$  The problem (\ref{regularizedb1}) can be rewritten as \begin{align} 
w(x,t)=& \frac{1}{\sqrt{2\pi}}\int_{-\infty}^{+\infty}{ 
	\cosh \Big(k(\omega) x\Big) \hat{g}_\epsilon^\delta(\omega)  e^{i\omega t}d\omega} +
\frac{1}{\sqrt{2\pi}}\int_{-\infty}^{+\infty}{ \frac{\sinh \Big(k(\omega) x\Big)}{k(\omega)} \widehat{h}_\ep^\delta(\omega) e^{i\omega t}d\omega} \nn \\
&-    \frac{1}{\sqrt{2\pi}}\int_{-\infty}^{+\infty}\int_0^x  \frac{ \sinh {\Big (k(\omega)(x-z)\Big) }  }{k(\omega)} \widehat{f}_\epsilon(z,\omega,w(z,\omega))e^{i\omega t}dzd\omega   , \label{regularizedsolution2}
\end{align}
for $x\ge 0$, $t\in\mathbb{R}$.\\


The following lemma will show  that the regularized problem (\ref{regularizedsolution2}) is well-posed.

\begin{lemma}\label{uniqlemma} Let $\epsilon>0$, $\delta>0$ and $g^\delta,h^\delta\in L^2(\mathbb{R})$. Assume that $f\in L^\infty([0,1]\times\mathbb{R}\times\mathbb{R})$ satisfies the following condition 
	\begin{align}
	|f(x,t,v_1)-f(x,t,v_2)|\,\,\,\le\,\,\,& K|v_1-v_2|, \quad 0\le x \le 1, \quad t\in \mathbb{R}, \quad v_1, v_2\in \mathbb{R} \label{f2}
	\end{align}
	for a constant $K>0$ independent of $x$, $t$, $v_1$, $v_2$. Then Problem (\ref{regularizedsolution2}) has a   unique solution denoted by $u_\epsilon^\delta \in C([0,1];L^2(\mathbb{R}))$.  
\end{lemma} 

\begin{proof} Define 
	\begin{align*}
&	\Theta_{\ep,\delta}(w)(x,t):=  \frac{1}{\sqrt{2\pi}}\int_{-\infty}^{+\infty}{ 
		\cosh \Big(k(\omega) x\Big) \hat{g}_{\ep}^\delta(\omega)  e^{i\omega t}d\omega}  + \frac{1}{\sqrt{2\pi}}\int_{-\infty}^{+\infty}{ \frac{\sinh \Big(k(\omega) x\Big)}{k(\omega)} \widehat{h}_{\ep}^\delta(\omega) e^{i\omega t}d\omega}        \nn\\
& \hspace*{2.2cm}	-  \frac{1}{\sqrt{2\pi}}\int_{-\infty}^{+\infty}\int_0^x  \frac{ \sinh {\Big (k(\omega)(x-z)\Big) }  }{k(\omega)} \widehat{f}_{\ep}(z,\omega,w(z,\omega))e^{i\omega t}dzd\omega,  
	\end{align*}  for all    $w\in C([0,1];L^2(\mathbb{R}))$.\\
	
	We  show that the  problem (\ref{regularizedsolution2}) has a   unique solution by proving  that $\Theta_{\ep,\delta}$ has a unique fixed point in $C([0,1];L^2(\mathbb{R}))$.  For $w_1, w_2\in C([0,1];L^2(\mathbb{R}))$ and $0\le x\le1$, we {\color{black}will} show  the following estimate
	\begin{align}
	  \|\Theta_{\ep,\delta}^m(w_1)(x,.)-\Theta_{\ep,\delta}^m(w_2)(x,.)\|^2_{L^2(\mathbb{R})}  \le \,\,& \left( K\exp\left(\epsilon^{-\frac{\alpha}{2}} \cos\frac{\alpha\pi}{4}\right) \right)^{2m} \frac{x^m}{m!}\,\, \|w_1-w_2\|^2 _{C([0,1];L^2(\mathbb{R}))} , \label{qn}
	\end{align} for all integer numbers $m \ge 1$ .
 For  $m=1$, we have  
	\begin{align}
	\,\,&  \Theta_{\ep,\delta}(w_1)(x,t)-\Theta_{\ep,\delta}(w_2)(x,t) \nn\\
	= \,\,&  - \frac{1}{\sqrt{2\pi}}\int_{-\infty}^{+\infty}\int_0^x  \frac{ \sinh {\Big (k(\omega)(x-z)\Big) }  }{k(\omega)} \left[\widehat{f}_\epsilon(z,\omega,w_1(z,\omega)) - \widehat{f}_\epsilon(z,\omega,w_2(z,\omega)) \right]dz e^{i\omega t}d\omega \nn\\
	= \,\,& \mathcal{F}^{-1}\left( -\int_0^x  \frac{ \sinh {\Big (k(\omega)(x-z)\Big) }  }{k(\omega)} \left[\widehat{f}_\epsilon(z,\omega,w_1(z,\omega)) - \widehat{f}_\epsilon(z,\omega,w_2(z,\omega)) \right]dz\right)(t). \label{FourierTransa}
	\end{align} 
By  applying the Parseval's Theorem to (\ref{FourierTransa}),  we get 
	\begin{align}
	& \hspace*{3.3cm} \|\Theta_{\ep,\delta}(w_1)(x,.)-\Theta_{\ep,\delta}(w_2)(x,.)\|_{L^2(\mathbb{R})}^2 \nn\\
	= \,\,& \int_{-\infty}^{+\infty}   \left| \int_0^x  \frac{ \sinh {\Big (k(\omega)(x-z)\Big) }  }{k(\omega)}   
	\left[\widehat{f}_\epsilon(z,\omega,w_1(z,\omega)) - \widehat{f}_\epsilon(z,\omega,w_2(z,\omega)) \right]dz\right|^2   d\omega. \label{ApplyParseval}
	\end{align}
	
	Applying the H\"older inequality to (\ref{ApplyParseval}), we obtain
	\begin{align}
	\,\,& \left| \int_0^x  \frac{ \sinh {\Big (k(\omega)(x-z)\Big) }  }{k(\omega)}   
	\left[\widehat{f}_\epsilon(z,\omega,w_1(z,\omega)) - \widehat{f}_\epsilon(z,\omega,w_2(z,\omega))\right] dz\right|^2 \nn\\
	\le \,\,&  x   \int_0^x \left| \frac{ \sinh {\Big (k(\omega)(x-z)\Big) }  }{k(\omega)}\right|^2  
	\left|\widehat{f}_\epsilon(z,\omega,w_1(z,\omega)) - \widehat{f}_\epsilon(z,\omega,w_2(z,\omega))\right|^2 dz   \nn\\
	\le \,\,&  x   \int_0^x (x-z)^2 \exp\Big(2\Re(k(\omega))(x-z)\Big)  
	\left|\widehat{f}_\epsilon(z,\omega,w_1(z,\omega)) - \widehat{f}_\epsilon(z,\omega,w_2(z,\omega))\right|^2 dz ,    \nn
	\end{align}
	 where $$\left| \frac{ \sinh {\Big (k(\omega)(x-z)\Big) }  }{k(\omega)}\right|^2 \le (x-z)^2 \exp\Big(2\Re(k(\omega))(x-z)\Big)$$ by using the Lemma \ref{ineqlemma}. Therefore, 
	\begin{align}
	&\|\Theta_{\ep,\delta}(w_1)(x,.)-\Theta_{\ep,\delta}(w_2)(x,.)\|_{L^2(\mathbb{R})}^2 \nn\\
	\le \,\,&  x   \int_0^x \int_{-\infty}^{+\infty}   (x-z)^2 \exp\Big(2\Re(k(\omega))(x-z)\Big) 
	\left|\widehat{f}_\epsilon(z,\omega,w_1(z,\omega)) - \widehat{f}_\epsilon(z,\omega,w_2(z,\omega))\right|^2 d\omega dz     \nn\\
	\le \,\,&  x   \int_0^x \int_{-\frac{1}{\ep}}^{\frac{1}{\ep}}   (x-z)^2 \exp\Big(2\Re(k(\omega))(x-z)\Big)  
	\left|\widehat{f}(z,\omega,w_1(z,\omega)) - \widehat{f}(z,\omega,w_2(z,\omega))\right|^2 d\omega dz     \nn\\
	\le \,\,&  x^3  \exp\left(2x\ep^{-{\frac{\alpha}{2}}}\cos\frac{\alpha\pi}{4}\right)   \int_0^x   
	\norm{ \widehat{f}(z,.,w_1(z,.)) - \widehat{f}(z,.,w_2(z,.)) }_{L^2(\mathbb{R})}^2  dz     \nn\\
	\le \,\,&   \exp\left(2\ep^{-{\frac{\alpha}{2}}}\cos\frac{\alpha\pi}{4} \right)   \int_0^x   
	\norm{ {f}(z,.,w_1(z,.)) -  {f}(z,.,w_2(z,.)) }_{L^2(\mathbb{R})}^2  dz     \nn\\
	\le \,\,&  K^2  \exp\left(2\ep^{-{\frac{\alpha}{2}}}\cos\frac{\alpha\pi}{4} \right)   \int_0^x   
	\norm{ w_1(z,.) -  w_2(z,.) }_{L^2(\mathbb{R})}^2  dz,     \label{thetaprove1}  
	\end{align} where  the Lipschitz property (\ref{f2}) has been used.  This {\color{black} immediately implies} that
	\begin{align}
	\|\Theta_{\ep,\delta}(w_1)(x,.)-\Theta_{\ep,\delta}(w_2)(x,.)\|_{L^2(\mathbb{R})}^2 \le \,\,& K^2 \exp\left( 2 \epsilon^{-\frac{\alpha}{2}} \cos\frac{\alpha\pi}{4} \right) x \|w_1-w_2\|^2_{C([0,1];L^2(\mathbb{R}))}, \nn
	\end{align} 
	so (\ref{qn}) holds for $m=1$. Assume that (\ref{qn}) holds  for $m=j$, $j\ge 1$, i.e.,
	\begin{align}
	\|\Theta_{\ep,\delta}^j(w_1)(x,.)-\Theta_{\ep,\delta}^j(w_2)(x,.)\|^2_{L^2(\mathbb{R})} \le \left( K \exp\left( \epsilon^{-\frac{\alpha}{2}} \cos\frac{\alpha\pi}{4} \right) \right)^{2j} \frac{x^j}{j!}\,\, \|w_1-w_2\|^2 _{C([0,1];L^2(\mathbb{R}))}. \nn
	\end{align} 
	We are going to prove that (\ref{qn}) holds  for $m=j+1$. Indeed, it follows from (\ref{thetaprove1}) that
	\begin{align}
	&\|\Theta_{\ep,\delta}^{j+1}(w_1)(x,.)-\Theta_{\ep,\delta}^{j+1}(w_2)(x,.)\|_{L^2(\mathbb{R})}^2 \nn\\
	\le\,\, & K^2 \exp\left( 2 \epsilon^{-\frac{\alpha}{2}} \cos\frac{\alpha\pi}{4} \right) 
	\int_{0}^{x} \|\Theta_{\ep,\delta}^{j}(w_1)(z,.)-\Theta_{\ep,\delta}^{j}(w_2)(z,.)\|_ {L^2(\mathbb{R})}^2dz\nn\\
	\le \,\, & K^2 \exp\left( 2 \epsilon^{-\frac{\alpha}{2}} \cos\frac{\alpha\pi}{4} \right) \int_{0}^{x} \left( K \exp\left( \epsilon^{-\frac{\alpha}{2}} \cos\frac{\alpha\pi}{4} \right) \right)^{2j} \frac{z^j}{j!} \,\,\|w_1-w_2\|^2 _{C([0,1];L^2(\mathbb{R}))} dz \nn\\
	\le \,\, & \left( K \exp\left( \epsilon^{-\frac{\alpha}{2}} \cos\frac{\alpha\pi}{4} \right) \right)^{2(j+1)} \frac{ x^{j+1} }{(j+1)!} \,\, \|w_1-w_2\|^2 _{C([0,1];L^2(\mathbb{R}))}. \label{takingsup}
	\end{align}
By induction principle,	we conclude that (\ref{qn})  holds for all integer number{\color{black}s} $m \ge 1$. Now, by taking the supremum of (\ref{takingsup}) in the variable $x$, we derive
	\begin{align}
 	\|\Theta_{\ep,\delta}^m(w_1)-\Theta_{\ep,\delta}^m(w_2)\|^2 _{C([0,1];L^2(\mathbb{R}))}   \le \left( K \exp\left( \epsilon^{-\frac{\alpha}{2}} \cos\frac{\alpha\pi}{4} \right) \right)^{2m} \frac{1}{m!}\,\, \|w_1-w_2\|^2 _{C([0,1];L^2(\mathbb{R}))}. \label{qn2}
	\end{align} 
	
	On the other hand, because 
	\begin{align*}
	\lim\limits_{m\to +\infty} \left( K \exp\left( \epsilon^{-\frac{\alpha}{2}} \cos\frac{\alpha\pi}{4} \right) \right)^{2m} \frac{1}{m!} = 0.
	\end{align*}
{\color{black} It follows} from (\ref{qn2})  that there exists an integer number $m_* \ge 1$ such that $\Theta_{\ep,\delta}^{m_*}$ is a contraction mapping. Thus,  there exists a unique fixed point of $\displaystyle\Theta_{\ep,\delta}^{m_*} \in C([0,1];L^2(\mathbb{R}))$ which denoted by $u_\epsilon^\delta$, i.e., $$\Theta_{\ep,\delta}^{m_*}(u_\epsilon^\delta)=u_\epsilon^\delta.$$
	Hence $\displaystyle\Theta_{\ep,\delta}^{m_*} \left(\Theta_{\ep,\delta} (u_\epsilon^\delta)\right)=\Theta_{\ep,\delta} (u_\epsilon^\delta)$, i.e., $\displaystyle\Theta_{\ep,\delta} (u_\epsilon^\delta)$ is also a fixed point of $\displaystyle\Theta_{\ep,\delta}^{m_*}$ in $C([0,1];L^2(\mathbb{R}))$. This implies that $\displaystyle\Theta_{\ep,\delta} (u_\epsilon^\delta) = u_\epsilon^\delta$, due to the uniqueness of the fixed point.
\end{proof} 

\subsection{$L^2$ estimate }

From now on,  let  $\delta>0$ be the error level and  $g^\delta, h^\delta \in L^2(\mathbb{R})$ be the measured data satisfy  (\ref{noisyassumption}).  Let $\ep:=\ep(\delta)>0$ be the regularization parameter and let $u_\epsilon^\delta$ be the regularized solution of (\ref{regularizedsolution2}) respectively. 

\begin{theorem}\label{theorem1} Let  $f$ be {\color{black}defined by} Lemma \ref{uniqlemma} and assume that the problem (\ref{exsolution}) has a  unique (exact) solution $u$ such that
	\begin{align}
	\int_{-\infty}^{+\infty}  \exp\Big(2(1-x)\Re(k(\omega)) \Big) \left|\hat{u}(x,\omega)  \right|^2 d\omega < M_1^2, \quad 0\le x <1, \label{assumptionu}
	\end{align}
	for $M_1>0$,  
	 then
	\begin{align}
	 \|u_{\epsilon}^\delta(x,.)-u(x,.)\|_{L^2(\mathbb{R})}
	\le     C_1  \exp\left( x \epsilon^{-\frac{\alpha}{2}} \cos\frac{\alpha\pi}{4}\right) \delta    + C_1  \exp\Big( (x-1)\ep^{-{\frac{\alpha}{2}}} \cos{\frac{\alpha \pi}{4}} \Big) \label{mainresult1} 
	\end{align}
where $C_1= 4e^{K^2} + 2M_1e^{K^2}$. As a consequence, if we choose $\displaystyle \epsilon= \left( \frac{\cos\frac{\alpha\pi}{4}}{\ln  \frac{1}{\delta}  } \right)^{\frac{2}{\alpha}}$ then 
	\begin{align}
	 \|u_{\epsilon}^\delta(x,.)-u(x,.)\|_{L^2(\mathbb{R})}
	\le   C_1 \delta^{1-x},\quad 0\le x <1. \label{mainresult2}
	\end{align}  
\end{theorem}

\textcolor{black} {
\begin{remark}
	If $f=0$ and $h=0$ then since \eqref{solution3}, we have
	\begin{align}
		\widehat{u}(x,\omega)&= \cosh \Big(k(\omega) x\Big) \widehat{g}(\omega) .
	\end{align}
	Then the left-hand side of \eqref {assumptionu} is 
	\begin{align}
	\int_{-\infty}^{+\infty}  \exp\Big(2(1-x)\Re(k(\omega)) \Big) \left|\hat{u}(x,\omega)  \right|^2 d\omega
	 	&=  	\int_{-\infty}^{+\infty}  \exp\Big(2(1-x)\Re(k(\omega)) \Big) |\cosh \Big(k(\omega) x\Big) |^2 \left|\widehat{g}(\omega)   \right|^2 d\omega\nn\\
		&\le 	\int_{-\infty}^{+\infty}  \exp\Big(2\Re(k(\omega)) \Big) \left|\widehat{g}(\omega)   \right|^2 d\omega
	\end{align}
	where we have used $ |\cosh (z)|  \le e^{\Re{(z)}}$ since Lemma 2.1.	
	Moreover, we get
	\begin{align}
 k(\omega) \int_0^1 	\widehat{u} (z,\omega) dz=  \frac{e^{k(\omega) } - e^{-k(\omega )} }{2}  \widehat{g}(\omega) .
	\end{align}
	This implies that
	\[
	\widehat{u}(1,\omega)+  k(\omega) \int_0^1 	\widehat{u} (z,\omega) dz= e^{k(\omega)} \widehat{g}(\omega)
	\]
	Hence, we get
	\begin{align}
	\int_{-\infty}^{+\infty}  \exp\Big(2\Re(k(\omega)) \Big) \left|\widehat{g}(\omega)   \right|^2 d\omega & \le 2\int_{-\infty}^{+\infty}  |	\widehat{u}(1,\omega)|^2 d\omega+ 2 \int_{-\infty}^{+\infty}   |\omega|^\alpha 	\int_0^1 |\widehat{u} (z,\omega)|^2 dz d\omega\nn\\
	&=2\int_{-\infty}^{+\infty}  |	\widehat{u}(1,\omega)|^2 d\omega+ 2 \int_0^1 \int_{-\infty}^{+\infty} |\omega|^\alpha |\widehat{u} (z,\omega)|^2 d\omega dz\nn\\
	&=2 \|u(1.,)\|^2_{L^2(\mathbb R)}+   2 \|u \|^2_{L^2(0,1; H^{\frac{\alpha}{2}}(\mathbb R))} . \label{es1}
	\end{align}
	Thus, if $u \in L^2(0,1; H^{\frac{\alpha}{2}}(\mathbb R)) $ then \eqref{es1} holds.  Therefore, we can say that the condition $ \eqref {assumptionu} $ is nature and makes sense.
\end{remark}
}
\begin{proof} We only consider $0\le x < 1$ throughout this proof.  We  recall that 
\begin{align}
&	u_\ep^\delta(x,t) =  \frac{1}{\sqrt{2\pi}}\int_{-\infty}^{+\infty}{ 
		\cosh \Big(k(\omega) x\Big) \hat{g}_{\ep}^\delta(\omega)  e^{i\omega t}d\omega}  + \frac{1}{\sqrt{2\pi}}\int_{-\infty}^{+\infty}{ \frac{\sinh \Big(k(\omega) x\Big)}{k(\omega)} \widehat{h}_{\ep}^\delta(\omega) e^{i\omega t}d\omega}        \nn\\
& \hspace*{1.9cm}	-  \frac{1}{\sqrt{2\pi}}\int_{-\infty}^{+\infty}\int_0^x  \frac{ \sinh {\Big (k(\omega)(x-z)\Big) }  }{k(\omega)} \widehat{f}_{\ep}(z,\omega,u_\ep^\delta(z,\omega))e^{i\omega t}dzd\omega  \label{ued}
	\end{align}
and
\begin{align}
&	u(x,t) =  \frac{1}{\sqrt{2\pi}}\int_{-\infty}^{+\infty}{ 
		\cosh \Big(k(\omega) x\Big) \hat{g}(\omega)  e^{i\omega t}d\omega}  + \frac{1}{\sqrt{2\pi}}\int_{-\infty}^{+\infty}{ \frac{\sinh \Big(k(\omega) x\Big)}{k(\omega)} \widehat{h}(\omega) e^{i\omega t}d\omega}        \nn\\
& \hspace*{1.7cm}	-  \frac{1}{\sqrt{2\pi}}\int_{-\infty}^{+\infty}\int_0^x  \frac{ \sinh {\Big (k(\omega)(x-z)\Big) }  }{k(\omega)} \widehat{f}(z,\omega,u(z,\omega))e^{i\omega t}dzd\omega. \label{u}  
	\end{align}
 Applying  Lemma \ref{uniqlemma}, there   exists a unique function $u_\epsilon \in C([0,1];L^2(\mathbb{R}))$ satisfying 
\begin{align}
&	u_\ep(x,t) =  \frac{1}{\sqrt{2\pi}}\int_{-\infty}^{+\infty}{ 
		\cosh \Big(k(\omega) x\Big) \hat{g}_{\ep}(\omega)  e^{i\omega t}d\omega}  + \frac{1}{\sqrt{2\pi}}\int_{-\infty}^{+\infty}{ \frac{\sinh \Big(k(\omega) x\Big)}{k(\omega)} \widehat{h}_{\ep}(\omega) e^{i\omega t}d\omega}        \nn\\
&	\hspace*{1.9cm} -  \frac{1}{\sqrt{2\pi}}\int_{-\infty}^{+\infty}\int_0^x  \frac{ \sinh {\Big (k(\omega)(x-z)\Big) }  }{k(\omega)} \widehat{f}_{\ep}(z,\omega,u_\ep(z,\omega))e^{i\omega t}dzd\omega  \label{ue}
	\end{align}
 where  $\displaystyle \Big[\widehat{g}_\epsilon(\omega) \quad \widehat{h}_\epsilon(\omega) \quad \widehat{f}_\epsilon(z,\omega,w(z,\omega)) \Big] := \Big[\widehat{g}(\omega) \quad \widehat{h}(\omega) \quad \widehat{f}(z,\omega,w(z,\omega)) \Big]\chi_{\left[-\frac{1}{\epsilon};\frac{1}{\epsilon}  \right]}(\omega)$. In order to establish an estimate for $\displaystyle \|u_{\epsilon}^\delta(x,.)-u(x,.)\|_{L^2(\mathbb{R})}$, we  introduce the quantity $\mathcal{P}_\ep (u)$ defined as follows 
\begin{align}
&	\mathcal{P}_\ep(u)(x,t) =  \frac{1}{\sqrt{2\pi}}\int_{-\infty}^{+\infty}{ 
		\cosh \Big(k(\omega) x\Big) \hat{g}_\ep(\omega)  e^{i\omega t}d\omega}  + \frac{1}{\sqrt{2\pi}}\int_{-\infty}^{+\infty}{ \frac{\sinh \Big(k(\omega) x\Big)}{k(\omega)} \widehat{h}_\ep(\omega) e^{i\omega t}d\omega}        \nn\\
& \hspace*{2.2cm}	- \frac{1}{\sqrt{2\pi}}\int_{-\infty}^{+\infty}\int_0^x  \frac{ \sinh {\Big (k(\omega)(x-z)\Big) }  }{k(\omega)} \widehat{f}_\ep(z,\omega,u(z,\omega))e^{i\omega t}dzd\omega , \label{peu} 
	\end{align} 
where $u$ is the exact solution.
The triangle inequality {\color{black}shows} that  
	\begin{align}
	\|u_{\epsilon}^\delta(x,.)-u(x,.)\|^2_{L^2(\mathbb{R})}\le &\,\, 2 \|u_{\epsilon}^\delta(x,.)-{u_\epsilon}(x,.)\|^2_{L^2(\mathbb{R})}   +2 \|u_\ep(x,.)-u(x,.)\|^2_{L^2(\mathbb{R})}  . \label{pl2etrin1}
	\end{align}
Next, we divide  the proof  into three steps.
	
	\vspace*{0.4cm}
	
	\noindent {\bf Step 1:} Estimate for  $\|u_{\epsilon}^\delta(x,.)-{u_\epsilon}(x,.)\|^2_{L^2(\mathbb{R})}$. It follows from (\ref{ued}), and (\ref{ue}) that
\begin{align}
	\,\,& \hspace*{6.5cm} u_{\epsilon}^\delta(x,t)-{u_\epsilon}(x,t)\nn\\
	  = \,\,& \frac{1}{\sqrt{2\pi}}\int_{-\infty}^{+\infty}{ 
		\cosh \Big(k(\omega) x\Big) \left[\hat{g}_{\epsilon}^\delta(\omega)-\hat{g}_{\epsilon}(\omega) \right]  e^{i\omega t}d\omega}   + \frac{1}{\sqrt{2\pi}}\int_{-\infty}^{+\infty}{ \frac{\sinh \Big(k(\omega) x\Big)}{k(\omega)} \left[\hat{h}_{\epsilon}^\delta(\omega)-\hat{h}_{\epsilon}(\omega) \right] e^{i\omega t}d\omega}        \nn\\
& 	-  \frac{1}{\sqrt{2\pi}}\int_{-\infty}^{+\infty}\int_0^x  \frac{ \sinh {\Big (k(\omega)(x-z)\Big) }  }{k(\omega)} \left[ \widehat{f}_\epsilon(z,\omega,u_{\epsilon}^\delta(z,\omega)) - \widehat{f}_\epsilon(z,\omega,u_{\epsilon}(z,\omega)) \right]e^{i\omega t}dzd\omega   \nn
\end{align}	
which can be represented by the inverse Fourier transform as
\begin{align}
 & \hspace*{5.2cm} u_{\epsilon}^\delta(x,t)-{u_\epsilon}(x,t) \nn\\
 & \hspace*{0.2cm} =  \mathcal{F}^{-1}\left( \cosh \Big(k(\omega) x\Big) \left[\hat{g}_{\epsilon}^\delta(\omega)-\hat{g}_{\epsilon}(\omega) \right] + \frac{\sinh \Big(k(\omega) x\Big)}{k(\omega)} \left[\hat{h}_{\epsilon}^\delta(\omega)-\hat{h}_{\epsilon}(\omega) \right]\right)(t)  \nn\\
 & +  \mathcal{F}^{-1}\left(-\int_0^x  \frac{ \sinh {\Big (k(\omega)(x-z)\Big) }  }{k(\omega)} \left[ \widehat{f}_\epsilon(z,\omega,u_{\epsilon}^\delta(z,\omega)) - \widehat{f}_\epsilon(z,\omega,u_{\epsilon}(z,\omega)) \right] dz\right)(t) . \label{pl2step1a}
\end{align}	
By applying the Parseval's Theorem to (\ref{pl2step1a}), we derive 
\begin{align}
 & \hspace*{4.5cm} \|u_{\epsilon}^\delta(x,.)-{u_\epsilon}(x,.)\|^2_{L^2(\mathbb{R})}\nn\\
& \le   2\underbrace{\int_{-\infty}^{+\infty} \left| \cosh \Big(k(\omega) x\Big) \left[\hat{g}_{\epsilon}^\delta(\omega)-\hat{g}_{\epsilon}(\omega) \right] + \frac{\sinh \Big(k(\omega) x\Big)}{k(\omega)} \left[\hat{h}_{\epsilon}^\delta(\omega)-\hat{h}_{\epsilon}(\omega) \right]  \right|^2 d\omega}_{I_{1}}  \nn\\
&+2\underbrace{\int_{-\infty}^{+\infty} \left|     \int_0^x  \frac{ \sinh {\Big (k(\omega)(x-z)\Big) }  }{k(\omega)} \left[ \widehat{f}_\epsilon(z,\omega,u_{\epsilon}^\delta(z,\omega)) - \widehat{f}_\epsilon(z,\omega,u_{\epsilon}(z,\omega)) \right] dz \right|^2 d\omega}_{I_{2}}. \label{pl2step1b}
\end{align}
We continue to estimate  $I_{1}$ as follows:
	\begin{align}
	I_{1} \le \,\, &  2\int_{-\infty}^{+\infty} \left| \cosh \Big(k(\omega) x\Big)\right|^2 \left|\hat{g}_{\epsilon}^\delta(\omega)-\hat{g}_{\epsilon}(\omega) \right|^2 d\omega + 2\int_{-\infty}^{+\infty} \left| \frac{\sinh \Big(k(\omega) x\Big)}{k(\omega)} \right|^2 \left|\hat{h}_{\epsilon}^\delta(\omega)-\hat{h}_{\epsilon}(\omega) \right|^2 d\omega \nn \\
	\le \,\, &  2\int_{-\frac{1}{\ep}}^{\frac{1}{\ep}} \left| \cosh \Big(k(\omega) x\Big)\right|^2 \left|\hat{g} ^\delta(\omega)-\hat{g} (\omega) \right|^2 d\omega   + 2\int_{-\frac{1}{\ep}}^{\frac{1}{\ep}} \left| \frac{\sinh \Big(k(\omega) x\Big)}{k(\omega)} \right|^2 \left|\hat{h} ^\delta(\omega)-\hat{h} (\omega) \right|^2 d\omega \nn \\
	\le \,\, &  2\int_{-\frac{1}{\ep}}^{\frac{1}{\ep}} \exp\Big(2x\Re(k(\omega))\Big) \left|\hat{g} ^\delta(\omega)-\hat{g} (\omega) \right|^2 d\omega   + 2\int_{-\frac{1}{\ep}}^{\frac{1}{\ep}} x^2 \exp\Big(2x\Re(k(\omega))\Big) . \left|\hat{h} ^\delta(\omega)-\hat{h} (\omega) \right|^2 d\omega, \nn 
	\end{align} where $$\left| \cosh \Big(k(\omega) x\Big)\right|^2 \le \exp\Big(2x\Re(k(\omega))\Big),$$ and $$\left| \frac{\sinh \Big(k(\omega) x\Big)}{k(\omega)} \right|^2 \le x^2 \exp\Big(2x\Re(k(\omega))\Big) $$ by applying Lemma \ref{ineqlemma}. 
	On the other hand, for $\displaystyle |\omega| \le \frac{1}{\ep}$ we have $$\exp\Big(2x\Re(k(\omega))\Big) \le \exp\left(2x \epsilon^{-\frac{\alpha}{2}} \cos\frac{\alpha\pi}{4}\right).$$ Thus we get
	\begin{align}
	I_{1} \le \,\,& 2 \exp\left(2x \epsilon^{-\frac{\alpha}{2}} \cos\frac{\alpha\pi}{4}\right) \| \hat{g}^\delta -\hat{g} \|_{L^2(\mathbb{R})}^2  +  2x^2 \exp\left(2x \epsilon^{-\frac{\alpha}{2}} \cos\frac{\alpha\pi}{4}\right) \| \hat{h}^\delta -\hat{h}\|_{L^2(\mathbb{R})}^2 .\nn
	\end{align} 
	Consequently, we can obtain 
	\begin{align}
	I_{1}  \le 4 \exp\left(2x \epsilon^{-\frac{\alpha}{2}} \cos\frac{\alpha\pi}{4}\right)\delta^2,   \label{pl2step1c}
	\end{align} where the fact that $x \le 1$ and the condition (\ref{noisyassumption}) have been used. 
	Applying the H\"older's inequality and  Lemma \ref{ineqlemma}, it yields 
we obtain
	\begin{align}
  & \hspace*{0.3cm}	I_{2} \le  \int_{-\infty}^{+\infty}   x   \int_0^x \left| \frac{ \sinh {\Big (k(\omega)(x-z)\Big) }  }{k(\omega)} \right|^2  \left| \widehat{f}_\epsilon(z,\omega,u_{\epsilon}^\delta(z,\omega)) - \widehat{f}_\epsilon(z,\omega,u_{\epsilon}(z,\omega)) \right|^2 dz  d\omega\nn\\
	 &   \le  x  \int_0^x  \int_{-\frac{1}{\epsilon}}^{\frac{1}{\epsilon}}  (x-z)^2\exp\Big( 2(x-z)\Re(k(\omega)) \Big) \left|\widehat{f}(z,\omega,u_{\epsilon}^\delta(z,\omega)) - \widehat{f}(z,\omega,u_{\epsilon}(z,\omega))  \right|^2d\omega dz\nn .
	\end{align} 
{\color{black}It} follows from $x(x-z)^2\le 1$ and $\exp\Big(2(x-z)\Re(k(\omega))\Big) \le \exp\left(2(x-z) \epsilon^{-\frac{\alpha}{2}} \cos\frac{\alpha\pi}{4}\right)$
for $0\le z\le x\le 1$, $\displaystyle |\omega| \le \frac{1}{\ep}$  that
\begin{align}
\,\,&  	I_{2}  
	\le  \exp\left(2x\ep^{-{\frac{\alpha}{2}}} \cos{\frac{\alpha \pi}{4}}\right)   \int_0^x \exp\left(-2z\ep^{-{\frac{\alpha}{2}}} \cos{\frac{\alpha \pi}{4}}\right) \times \nn\\
	& \hspace*{0.2cm} \times \int_{-\frac{1}{\epsilon}}^{\frac{1}{\epsilon}}        \left|\widehat{f}(z,\omega,u_{\epsilon}^\delta(z,\omega)) - \widehat{f}(z,\omega,u_{\epsilon}(z,\omega))  \right|^2 d\omega  dz.\nn
\end{align}
Using Paserval's identity and the Lipschitz condition (\ref{f2}), we obtain
\begin{align}
& \int_{-\frac{1}{\epsilon}}^{\frac{1}{\epsilon}}        \left|\widehat{f}(z,\omega,u_{\epsilon}^\delta(z,\omega)) - \widehat{f}(z,\omega,u_{\epsilon}(z,\omega))  \right|^2 d\omega  dz \nn\\
  \le \,\,&   \norm{\widehat{f}(z,.,u_{\epsilon}^\delta(z,.)) - \widehat{f}(z,.,u_{\epsilon}(z,.))}_{L^2(\mathbb{R})}^2   \nn\\
= \,\,& \norm{ {f}(z,.,u_{\epsilon}^\delta(z,.)) -  {f}(z,.,u_{\epsilon}(z,.))}_{L^2(\mathbb{R})}^2 \nn\\
\le \,\,& K^2 \norm{ u_{\epsilon}^\delta(z,.) -  u_{\epsilon}(z,.)}_{L^2(\mathbb{R})}^2 . \nn
\end{align} 
 Therefore, we derive
\begin{align}
&	I_{2}  
	\le    K^2 \exp\left(2x\ep^{-{\frac{\alpha}{2}}} \cos{\frac{\alpha \pi}{4}}\right)  \int_0^x \exp\left(-2z\ep^{-{\frac{\alpha}{2}}} \cos{\frac{\alpha \pi}{4}}\right) K^2 \norm{ u_{\epsilon}^\delta(z,.) -  u_{\epsilon}(z,.)}_{L^2(\mathbb{R})}^2  dz.\nn\\
& \hspace*{3.5cm}	\le \,\, K^2 \exp\left( 2x\ep^{-{\frac{\alpha}{2}}} \cos{\frac{\alpha \pi}{4}} \right) \int_0^x \mathcal{Y}_1(z)dz, \label{pl2step1d}
	\end{align}
where we put  $$\displaystyle \mathcal{Y}_1(z):=\exp\left(-2z\epsilon^{-\frac{\alpha}{2}} \cos{\frac{\alpha \pi}{4}}\right) \| u_{\epsilon}^\delta(z,.) - u_{\epsilon}(z,.) \|_{L^2(\mathbb{R})}^2 \ge 0$$ for $0\le z\le x$.	It follows from (\ref{pl2step1b}), (\ref{pl2step1c}), and (\ref{pl2step1d}) that
	\begin{align}
	& \hspace*{1.9cm} \|u_{\epsilon}^\delta(x,.)-{u_\epsilon}(x,.)\|^2_{L^2(\mathbb{R})}  \le   2I_1+2I_2 \nn\\
	  \le  \,\,& 8  \exp\left(2x \epsilon^{-\frac{\alpha}{2}} \cos\frac{\alpha\pi}{4} \right)\delta^2   + 2K^2  \exp\left(2x\epsilon^{-\frac{\alpha}{2}} \cos{\frac{\alpha \pi}{4}}\right) \int_0^x  \mathcal{Y}_1(z) dz,\label{pl2gwa}
	\end{align}
	which is equivalent to 
	\begin{align}
	\mathcal{Y}_1(x) \le 8\delta^2     + 2K^2 \int_0^x  \mathcal{Y}_1(z) dz. \label{pl2step1e}
	\end{align}
	Applying the Gronwall's inequality to (\ref{pl2step1e}), we obtain
	\begin{align}
	\mathcal{Y}_1(x) \le 8\delta^2 e^{2K^2x}   .\nn 
	\end{align}
	Hence,
	\begin{align}
	\|u_{\epsilon}^\delta(x,.)-{u_\epsilon}(x,.)\|^2_{L^2(\mathbb{R})} &\,\, \le 8  e^{2K^2 }  \exp\left(2x \epsilon^{-\frac{\alpha}{2}} \cos\frac{\alpha\pi}{4}\right) \delta^2  .\label{pl2step1e}
	\end{align}
	
\vspace*{0.3cm}
	
\noindent {\bf Step 2:} Estimate for  $\|u_\ep(x,.)-u(x,.)\|^2_{L^2(\mathbb{R})}$. We have
\begin{align}
	\|u_\ep(x,.)-u(x,.)\|^2_{L^2(\mathbb{R})}\le &\,\,    2\|u_\ep(x,.)-\mathcal{P}_\ep(u)(x,.)\|^2_{L^2(\mathbb{R})}  + 2  \|\mathcal{P}_\ep(u)(x,.)-u(x,.)\|^2_{L^2(\mathbb{R})}. \label{pl2etrin2}
	\end{align}  We    split this step into three sub-steps. The first one is estimating $\|u_\ep(x,.)-\mathcal{P}_\ep(u)(x,.)\|^2_{L^2(\mathbb{R})}$ and the second one is estimating $\|\mathcal{P}_\ep(u)(x,.)-u(x,.)\|^2_{L^2(\mathbb{R})}$. The last one will combine two first ones to obtain an estimate for $\|u_\ep(x,.)-u(x,.)\|^2_{L^2(\mathbb{R})}$. \\

\noindent {\bf Sub-step 2a:} Estimate for $\|u_\ep(x,.)-\mathcal{P}_\ep(u)(x,.)\|^2_{L^2(\mathbb{R})}$. By subtracting the equations (\ref{ue}) and (\ref{peu}), we derive
\begin{align}
  u_\ep(x,t)-\mathcal{P}_\ep(u)(x,t)   
 =    \mathcal{F}^{-1}\left(-\int_0^x  \frac{ \sinh {\Big (k(\omega)(x-z)\Big) }  }{k(\omega)} \left[ \widehat{f}_\epsilon(z,\omega,u_{\epsilon} (z,\omega)) - \widehat{f}_\epsilon(z,\omega,u(z,\omega)) \right] dz\right)(t). \nn
\end{align}	
By using the same proof as for (\ref{pl2step1d}), we obtain
\begin{align}
\,\,& \|u_\ep(x,.)-\mathcal{P}_\ep(u)(x,.)\|^2_{L^2(\mathbb{R})} \le   K^2 \exp\left( 2x\ep^{-{\frac{\alpha}{2}}} \cos{\frac{\alpha \pi}{4}} \right) \int_0^x \mathcal{Z}_1(z)dz, \label{pl2step2aa}
\end{align} where $\displaystyle \mathcal{Z}_1(z):=\exp\left(-2z\epsilon^{-\frac{\alpha}{2}} \cos{\frac{\alpha \pi}{4}}\right) \| u_{\epsilon} (z,.) - u (z,.) \|_{L^2(\mathbb{R})}^2 \ge 0$ for all $0\le z\le x$.\\

\noindent {\bf Sub-step 2b:} Estimate for  $\|\mathcal{P}_\ep(u)(x,.)-u(x,.)\|^2_{L^2(\mathbb{R})}$. By subtracting the equations (\ref{ue}) and (\ref{peu}), we obtain
\begin{align}
	 & \hspace*{6.3cm} \mathcal{P}_\ep(u)(x,t)-u(x,t)\nn\\
	&  =  \frac{1}{\sqrt{2\pi}}\int_{-\infty}^{+\infty}{ 
		\cosh \Big(k(\omega) x\Big) \left[\hat{g}_{\epsilon} (\omega)-\hat{g} (\omega) \right]  e^{i\omega t}d\omega}   + \frac{1}{\sqrt{2\pi}}\int_{-\infty}^{+\infty}{ \frac{\sinh \Big(k(\omega) x\Big)}{k(\omega)} \left[\hat{h}_{\epsilon} (\omega)-\hat{h} (\omega) \right] e^{i\omega t}d\omega}        \nn\\
& \hspace*{1.3cm} 	-  \frac{1}{\sqrt{2\pi}}\int_{-\infty}^{+\infty}\int_0^x  \frac{ \sinh {\Big (k(\omega)(x-z)\Big) }  }{k(\omega)} \left[ \widehat{f}_\epsilon(z,\omega,u(z,\omega)) - \widehat{f}(z,\omega,u(z,\omega)) \right]e^{i\omega t}dzd\omega,   \nn
\end{align}	
which can be represented by the inverse Fourier transform as follows
\begin{align}
 & \hspace*{4.7cm} \mathcal{P}_\ep(u)(x,t)-u(x,t) \nn\\
& \hspace*{0.3cm} =   \mathcal{F}^{-1}\left( \cosh \Big(k(\omega) x\Big) \left[\hat{g}_{\epsilon} (\omega)-\hat{g} (\omega) \right] + \frac{\sinh \Big(k(\omega) x\Big)}{k(\omega)} \left[\hat{h}_{\epsilon} (\omega)-\hat{h} (\omega) \right]\right)(t)  \nn\\
 & +  \mathcal{F}^{-1}\left(-\int_0^x  \frac{ \sinh {\Big (k(\omega)(x-z)\Big) }  }{k(\omega)} \left[ \widehat{f}_\epsilon(z,\omega,u(z,\omega)) - \widehat{f}(z,\omega,u(z,\omega)) \right] dz\right)(t) \nn\\
& \hspace*{3.7cm} =    \mathcal{F}^{-1}\left( -   \hat{u}(x,\omega)   \chi_{\mathbb{R} \setminus \left[-\frac{1}{\epsilon};\frac{1}{\epsilon}\right]}(\omega) \right).
 \label{pl2nopeumu}
\end{align}	 Taking $L^2(\mathbb{R})$-norm of $\mathcal{P}_\ep(u)(x,t)-u(x,t)$ and applying the Parseval's identity, we get 
\begin{align}
  \norm{\mathcal{P}_\ep(u)(x,.)-u(x,.)}^2_{L^2(\mathbb{R})}  =  \int_{\mathbb{R} \setminus \left[-\frac{1}{\epsilon};\frac{1}{\epsilon}\right] } \left|\hat{u}(x,\omega)  \right|^2 d\omega.\nn 
\end{align}
Therefore,
\begin{align}
& \hspace*{3.8cm}  \norm{\mathcal{P}_\ep(u)(x,.)-u(x,.)}^2_{L^2(\mathbb{R})} \nn\\ 
& =    \int_{\mathbb{R} \setminus  \left[-\frac{1}{\epsilon};\frac{1}{\epsilon}\right] }  
\exp\Big(2(x-1)\Re(k(\omega)) \Big)
 \left|\hat{u}(x,\omega)  \right|^2 \exp\Big(2(1-x)\Re(k(\omega)) \Big) d\omega \nn\\
& \le   \exp\Big(2(x-1)\ep^{-{\frac{\alpha}{2}}} \cos{\frac{\alpha \pi}{4}} \Big) \int_{\mathbb{R} \setminus \left[-\frac{1}{\epsilon};\frac{1}{\epsilon}\right] }  \exp\Big(2(1-x)\Re(k(\omega)) \Big) \left|\hat{u}(x,\omega)  \right|^2  d\omega, \nn 
\end{align} 
where  we have used that  $$\displaystyle \exp\Big(2(x-1)\Re(k(\omega)) \Big)=\exp\Big(2(x-1) |w|^{ \frac{\alpha}{2}}\cos{\frac{\alpha \pi}{4}} \Big)  \le  \exp\Big(2(x-1)\ep^{-{\frac{\alpha}{2}}} \cos{\frac{\alpha \pi}{4}} \Big),$$ for all $ \displaystyle |\omega| \le \frac{1}{\ep} $.
On the other hand, it follows from the assumption (\ref{assumptionu2}) that  $$\int_{\mathbb{R} \setminus \left[-\frac{1}{\epsilon};\frac{1}{\epsilon}\right] }  \exp\Big(2(1-x)\Re(k(\omega)) \Big) \left|\hat{u}(x,\omega)  \right|^2  d\omega \le M_1^2.$$
Hence, 
\begin{align}
  \norm{\mathcal{P}_\ep(u)(x,.)-u(x,.)}^2_{L^2(\mathbb{R})}  
\le M_1^2  \exp\Big(2(x-1)\ep^{-{\frac{\alpha}{2}}} \cos{\frac{\alpha \pi}{4}} \Big). \label{pl2step2ba}
\end{align}

\noindent {\bf Sub-step 2c:} Estimate   for $\|u_\ep(x,.)-u(x,.)\|^2_{L^2(\mathbb{R})}$. Combining (\ref{pl2etrin2}),   (\ref{pl2step2aa}) and  (\ref{pl2step2ba}), we derive 
\begin{align}
   &  \|u_\ep(x,.)-u(x,.)\|^2_{L^2(\mathbb{R})}  \le 2M_1^2 \exp\Big(2(x-1)\ep^{-{\frac{\alpha}{2}}} \cos{\frac{\alpha \pi}{4}} \Big)\nn\\
   & \hspace*{1.2cm}  +   2K^2  \exp\left(2x\ep^{-{\frac{\alpha}{2}}} \cos{\frac{\alpha \pi}{4}}\right)   \int_0^x \mathcal{Z}_1(z) dz.
\end{align}
This implies that
\begin{align}
 \mathcal{Z}_1(x)\le & 2 M_1^2 \exp\Big(-2\ep^{-{\frac{\alpha}{2}}} \cos{\frac{\alpha \pi}{4}} \Big)   +   2K^2      \int_0^x \mathcal{Z}_1(z)  dz. \label{l2step2f}
\end{align} Therefore, 
$$\mathcal{Q}_1(x)\le 2M_1^2 \exp\Big(-2\ep^{-{\frac{\alpha}{2}}} \cos{\frac{\alpha \pi}{4}} \Big) e^{2K^2x}$$ by applying Gronwall's inequality to (\ref{l2step2f}). Thus,
	\begin{align}
	 \|u_\ep(x,.)-u(x,.)\|^2_{L^2(\mathbb{R})}  &\,\,\le 2M_1^2e^{2K^2 }  \exp\Big(2(x-1)\ep^{-{\frac{\alpha}{2}}} \cos{\frac{\alpha \pi}{4}} \Big)  . \label{pl2step2ca}
	\end{align}

	\noindent {\bf Step 3:}  Estimate for  $\|u_{\epsilon}^\delta(x,.)-u(x,.)\|_{L^2(\mathbb{R})}$. Since (\ref{pl2etrin1}), (\ref{pl2step1e}) and (\ref{pl2step2ca}), we obtain
	\begin{align}
 \,\,&	\|u_{\epsilon}^\delta(x,.)-u(x,.)\|^2_{L^2(\mathbb{R})}  
	 	\le  16  e^{2K^2 }  \exp\left(2 x \epsilon^{-\frac{\alpha}{2}} \cos\frac{\alpha\pi}{4}\right) \delta^2  \,\,+\,\,  4M_1^2e^{2K^2 } \exp\Big(2(x-1)\ep^{-{\frac{\alpha}{2}}} \cos{\frac{\alpha \pi}{4}} \Big){\color{black}.}   \nn  
	\end{align} {\color{black}S}o
	\begin{align}
	\|u_{\epsilon}^\delta(x,.)-u(x,.)\|_{L^2(\mathbb{R})}
	\le     4 e^{ K^2 }  \exp\left( x \epsilon^{-\frac{\alpha}{2}} \cos\frac{\alpha\pi}{4}\right) \delta    + 2M_1e^{K^2}  \exp\Big( (x-1)\ep^{-{\frac{\alpha}{2}}} \cos{\frac{\alpha \pi}{4}} \Big)  ,   \label{error}  
	\end{align} i.e., (\ref{mainresult1}) {\color{black} is} proved. In addition, substituting $\displaystyle \epsilon= \left( \frac{\cos\frac{\alpha\pi}{4}}{\ln  \frac{1}{\delta}  } \right)^{\frac{2}{\alpha}}$ into (\ref{error}) we get
	\begin{align}
	   \|u_{\epsilon}^\delta(x,.)-u(x,.)\|_{L^2(\mathbb{R})}
	\le   C_1 \delta^{1-x} .
	\end{align} 
\end{proof}

\subsection{$H^p$ estimate }

In {\color{black} Theorem \ref{theorem1}}, an estimate of $\|u_{\epsilon}^\delta(x,.)-u(x,.)\|_{L^2(\mathbb{R})}$ {\color{black} was}  given  according to the a priori condition (\ref{assumptionu}). To obtain a result in $H^p(\mathbb{R})$,  we  assume that the exact solution $u$ satisfies the stronger condition (\ref{assumptionu2}).   Recall that $H^p(\mathbb{R})$ is defined in Section 2.    

\begin{theorem}\label{theorem2} Let  $f$ be {\color{black} defined by} Lemma \ref{uniqlemma}.  Suppose that there exist  constants $M_2>0$,\, $\displaystyle   \gamma >0$\,  and \,$\mu>\max\{4-\alpha;\,4p-\alpha\}$\, such that 
	\begin{align}
		\int_{-\infty}^{+\infty}   \exp\left(2(1+\gamma-x) |\omega|^{ \frac{\alpha+\mu }{2}} \cos\frac{\alpha\pi}{4}\right) \left|\hat{u}(x,\omega) \right|^2  d\omega < M_2^2, \qquad 0\le x< 1 . \label{assumptionu2}
	\end{align}  Let us choose   $\displaystyle \delta $  such that
\begin{align}
\displaystyle  \left( \frac{A+B}{\ln   \frac{1}{\delta} }\right)^\frac{2}{\alpha+\mu} < \min\left\{ 1;\, \left(\frac{B}{A}\right)^{\frac{1}{\frac{\alpha+\mu}{2}-q}};\, \left[\frac{\frac{\alpha+\mu}{2}\gamma  \cos\frac{\alpha\pi}{4}}{p}\right]^{\frac{2}{\alpha+\mu}} \right\} \label{deltachoosing}
\end{align}
and choose\, $ \displaystyle \epsilon=\left( \frac{A+B}{\ln   \frac{1}{\delta} }\right)^\frac{2}{\alpha+\mu}$\,
	then
\begin{align}
\|u_{\epsilon}^\delta(x,.)-u(x,.)\|_{H^p(\mathbb{R})}  
    \le    &\,\,    C_2\delta^{\frac{B}{A+B}}  , \label{result2b}
	\end{align} for $0\le x<1$
	where  
	$ \displaystyle A  = \frac{p}{2}  +1+ K^22^p , \,   \displaystyle B=\frac{1}{2}\gamma \cos\frac{\alpha \pi}{4}, \,  \displaystyle q  = \max\{\,2;\,\frac{\alpha}{2};\,2p\},\,  C_2= 4+2M_2.$ 
\end{theorem}

\vspace*{0.3cm}   
   
\begin{proof} First, we have 
	\begin{align}
		\|u_{\epsilon}^\delta(x,.)-u(x,.)\|^2_{H^p(\mathbb{R})}\le 2\|u_{\epsilon}^\delta(x,.)-{u_\epsilon}(x,.)\|^2_{H^p(\mathbb{R})}+2\|{u_\epsilon}(x,.)-u(x,.)\|^2_{H^p(\mathbb{R})}. \label{etrinhp}
	\end{align}
{\color{black} We split the proof into} three steps as follows\\
	
	\noindent {\bf Step 1:} Estimate $\|u_{\epsilon}^\delta(x,.)-{u_\epsilon}(x,.)\|^2_{H^p(\mathbb{R})}$.  From  (\ref{pl2step1a}), we get 
	\begin{align}
 & \hspace*{5.3cm} \|u_{\epsilon}^\delta(x,.)-{u_\epsilon}(x,.)\|^2_{H^p(\mathbb{R})}\nn\\
& \le   2\underbrace{\int_{-\infty}^{+\infty}(1+\omega^2)^p \left| \cosh \Big(k(\omega) x\Big) \left[\hat{g}_{\epsilon}^\delta(\omega)-\hat{g}_{\epsilon}(\omega) \right] + \frac{\sinh \Big(k(\omega) x\Big)}{k(\omega)} \left[\hat{h}_{\epsilon}^\delta(\omega)-\hat{h}_{\epsilon}(\omega) \right]  \right|^2 d\omega}_{J_{1}}  \nn\\
& +2\underbrace{\int_{-\infty}^{+\infty} (1+\omega^2)^p\left|     \int_0^x  \frac{ \sinh {\Big (k(\omega)(x-z)\Big) }  }{k(\omega)} \left[ \widehat{f}_\epsilon(z,\omega,u_{\epsilon}^\delta(z,\omega)) - \widehat{f}_\epsilon(z,\omega,u_{\epsilon}(z,\omega)) \right] dz \right|^2 d\omega}_{J_{2}}. \nn
\end{align}
	
The quantities $J_1$ and $J_2$  are  estimated by the same way as  Theorem \ref{theorem1}. Firstly, we note that $\displaystyle (1+\omega^2)^p \le \left(1+ \epsilon^{-2} \right)^p$ and $\displaystyle \exp\Big(2x\Re(k(\omega)) \Big) \le  \exp\Big(2x\ep^{-{\frac{\alpha}{2}}} \cos{\frac{\alpha \pi}{4}} \Big)$ for all $\displaystyle |\omega| \le \frac{1}{\epsilon}$, so
	\begin{align}
		& J_{1} \le    2  \left(1+ \epsilon^{-2} \right)^p \exp\left(2x \epsilon^{-\frac{\alpha}{2}} \cos\frac{\alpha\pi}{4}\right)    \| \hat{g}^\delta -\hat{g} \|_{L^2(\mathbb{R})}^2   \nn\\
		& \hspace*{0.1cm} +  2 \left(1+ \epsilon^{-2} \right)^p  \exp\left(2x \epsilon^{-\frac{\alpha}{2}} \cos\frac{\alpha\pi}{4}\right)  \| \hat{h}^\delta -\hat{h}\|_{L^2(\mathbb{R})}^2 {\color{black}.} \nn
	\end{align} {\color{black}This} implies 
	\begin{align}
		\|u_{\epsilon}^\delta(x,.)-{u_\epsilon}(x,.)\|^2_{H^p(\mathbb{R})} 
		\le   4 \left(1+ \epsilon^{-2} \right)^p \exp\left( 2x \epsilon^{-\frac{\alpha}{2}} \cos\frac{\alpha\pi}{4}\right)  \delta  ^2,  \label{phpstep1a}
	\end{align} where the condition (\ref{noisyassumption}) {\color{black} has been} used.
	 Secondly, it follows from  $$\displaystyle \exp\Big(2(x-z)\Re(k(\omega)) \Big) \le  \exp\Big(2(x-z)\ep^{-{\frac{\alpha}{2}}} \cos{\frac{\alpha \pi}{4}} \Big),$$ for all $\displaystyle |\omega| \le \frac{1}{\epsilon}$ and $z \le x$ that 
	\begin{align}
		J_{2} \le \,\,&  \left(1+ \epsilon^{-2} \right)^p K^2  \exp\left(2x\ep^{-{\frac{\alpha}{2}}} \cos{\frac{\alpha \pi}{4}}\right)      \int_0^x \exp\left(-2z\ep^{-{\frac{\alpha}{2}}} \cos{\frac{\alpha \pi}{4}}\right) \norm{ u_{\epsilon}^\delta(z,.) -  u_{\epsilon}(z,.)}_{L^2(\mathbb{R})}^2    dz,\nn
	\end{align} where  the Lipschitz property (\ref{f2}) {\color{black} has been used}.
	 Moreover, because of the fact that   $$\norm{ u_{\epsilon}^\delta(z,.) -  u_{\epsilon}(z,.)}_{L^2(\mathbb{R})}^2 \le \norm{ u_{\epsilon}^\delta(z,.) -  u_{\epsilon}(z,.)}_{H^p(\mathbb{R})}^2,$$
 we obtain
	\begin{align}
	 J_2  	\le  & K^2 \left(1+ \epsilon^{-2} \right)^p    \exp\left(2x\ep^{-{\frac{\alpha}{2}}} \cos{\frac{\alpha \pi}{4}}\right)   \int_0^x \mathcal{Y}_2(z)   dz , \label{phpstep1b}
	\end{align}
	where $$\displaystyle \mathcal{Y}_2(z)=\exp\left(-2z\epsilon^{-\frac{\alpha}{2}} \cos{\frac{\alpha \pi}{4}}\right) \| u_{\epsilon}^\delta(z,.) - u_{\epsilon}(z,.) \|_{H^p(\mathbb{R})}^2 \ge 0,$$ for all $0\le z\le x$.
	Now, by associating   (\ref{phpstep1a}) with (\ref{phpstep1b}), we derive
	\begin{align}
		  & \hspace*{5cm} \|u_{\epsilon}^\delta(x,.)-{u_\epsilon}(x,.)\|^2_{H^p(\mathbb{R})}  \nn\\
	&	\le    8 \left(1+ \epsilon^{-2} \right)^p \exp\left( 2x \epsilon^{-\frac{\alpha}{2}} \cos\frac{\alpha\pi}{4}\right)  \delta  ^2   +2K^2 \left(1+ \epsilon^{-2} \right)^p    \exp\left(2x\ep^{-{\frac{\alpha}{2}}} \cos{\frac{\alpha \pi}{4}}\right)   \int_0^x \mathcal{Y}_2(z)   dz .\nn
	\end{align}
	This implies that
	\begin{align}
		\mathcal{Y}_2(x) \le     8 \left(1+ \epsilon^{-2} \right)^p   \delta  ^2   +2K^2 \left(1+ \epsilon^{-2} \right)^p \int_0^x \mathcal{Y}_2(z)   dz.\nn
	\end{align}
	Applying the Gronwall's inequality, we get
	\begin{align}
		\mathcal{Y}_2(x) \le 8 \left(1+ \epsilon^{-2} \right)^p   \delta  ^2    \exp\left(2K^2 \left(1+ \epsilon^{-2} \right)^p x\right) {\color{black}.}  \nn 
	\end{align} {\color{black} Therefore,}
	 \begin{align}
		 \| u_{\epsilon}^\delta(x,.) - u_{\epsilon}(x,.) \|_{H^p(\mathbb{R})}^2  
	 \le   8 \left(1+ \epsilon^{-2} \right)^p \exp\left(  2 x \epsilon^{-\frac{\alpha}{2}} \cos\frac{\alpha\pi}{4}\right)     \exp\left(2K^2 \left(1+ \epsilon^{-2} \right)^px \right)  \delta  ^2.\label{phpstep1c}
	\end{align}
	
\vspace*{0.3cm}

	\noindent {\bf Step 2:} Estimate $\|{u_\epsilon}(x,.)-u(x,.)\|^2_{H^p(\mathbb{R})}$. We divide this step into three sub-steps since 
\begin{align}
	\|u_\ep(x,.)-u(x,.)\|^2_{H^p(\mathbb{R})}\le &\,\,    2\|u_\ep(x,.)-\mathcal{P}_\ep(u)(x,.)\|^2_{H^p(\mathbb{R})}  + 2  \|\mathcal{P}_\ep(u)(x,.)-u(x,.)\|^2_{H^p(\mathbb{R})}. \label{phpetrin2}
\end{align}	
	 The first and the second {\color{black} steps} are estimating $\|u_\ep(x,.)-\mathcal{P}_\ep(u)(x,.)\|^2_{H^p(\mathbb{R})}$ and  $\|\mathcal{P}_\ep(u)(x,.)-u(x,.)\|^2_{H^p(\mathbb{R})}$ respectively. The last sub-step will combine two first sub-steps to derive an estimate for $\|u_\ep(x,.)-u(x,.)\|^2_{H^p(\mathbb{R})}$.\\ 
	  
\noindent {\bf Sub-step 2a:}  Estimating $\|u_\ep(x,.)-\mathcal{P}_\ep(u)(x,.)\|^2_{H^p(\mathbb{R})}$.  It follows from (\ref{ue}) and (\ref{peu}) that
\begin{align}
\,\,& u_\ep(x,t)-\mathcal{P}_\ep(u)(x,t)  =    \mathcal{F}^{-1}\left(-\int_0^x  \frac{ \sinh {\Big (k(\omega)(x-z)\Big) }  }{k(\omega)} \left[ \widehat{f}_\epsilon(z,\omega,u_{\epsilon} (z,\omega)) - \widehat{f}_\epsilon(z,\omega,u(z,\omega)) \right] dz\right)(t). \nn
\end{align}
Hence,
\begin{align}
 & \hspace*{4.9cm} \|u_\ep(x,.)-\mathcal{P}_\ep(u)(x,.)\|^2_{H^p(\mathbb{R})} \nn\\
&	=  \int_{-\infty}^{+\infty} (1+\omega^2)^p \left|     \int_0^x  \frac{ \sinh {\Big (k(\omega)(x-z)\Big) }  }{k(\omega)} \left[ \widehat{f}_\epsilon(z,\omega,u(z,\omega)) - \widehat{f}_\epsilon(z,\omega,u_{\epsilon}(z,\omega)) \right] dz \right|^2 d\omega . \nn
	\end{align}
{\color{black} We note that } 
	\begin{align*}
	\left| \frac{ \sinh {\Big (k(\omega)(x-z)\Big) }  }{k(\omega)} \right|^2 \le (x-z)^2 \exp\Big( 2(x-z)\Re(k(\omega)) \Big){\color{black},}
	\end{align*}
{\color{black} and}  
\begin{align*}
	  \exp\Big( 2(x-z)\Re(k(\omega)) \Big) \le  \exp\left(2(x-z)\epsilon^{-\frac{\alpha}{2}} \cos{\frac{\alpha \pi}{4}}\right) \le  \exp\left(2(x-z)\epsilon^{-\frac{\alpha+\mu}{2}} \cos{\frac{\alpha \pi}{4}}\right), \quad |\omega| \le \frac{1}{\epsilon},
	\end{align*}
since \,\, $\ep^{-\frac{\alpha}{2}} < \epsilon^{-\frac{\alpha+\mu}{2}}$ \,\, for  $\displaystyle \frac{1}{\ep} > 1$ . By  a similar argument as in (\ref{phpstep1b}), we obtain
\begin{align}
	 \|u_\ep(x,.)-\mathcal{P}_\ep(u)(x,.)\|^2_{H^p(\mathbb{R})}  \le   K^2 \left(1+ \epsilon^{-2} \right)^p    \exp\left(2x\ep^{-{\frac{\alpha+\mu}{2}}} \cos{\frac{\alpha \pi}{4}}\right)   \int_0^x \mathcal{Z}_2(z)   dz , \label{phpstep2aa}
	\end{align} where $$\displaystyle \mathcal{Z}_2(z):=\exp\left(-2z\epsilon^{-\frac{\alpha+\mu}{2}} \cos{\frac{\alpha \pi}{4}}\right) \| u  (z,.) - u_{\epsilon}(z,.) \|_{H^p(\mathbb{R})}^2 \ge 0,$$ for all $0\le z\le x$. \\

\noindent {\bf Sub-step 2b:} Estimate $\|\mathcal{P}_\ep(u)(x,.)-u(x,.)\|^2_{H^p (\mathbb{R})}$. The equations (\ref{peu}) and (\ref{u}) show that 
\begin{align}
  \mathcal{P}_\ep(u)(x,t)-u(x,t) =
   \mathcal{F}^{-1}\left( -   \hat{u}(x,\omega)   \chi_{\mathbb{R} \setminus \left[-\frac{1}{\epsilon};\frac{1}{\epsilon}\right]}(\omega) \right)
 \nn
\end{align}	where $\hat{u}(x,\omega)$ is given in (\ref{solution3}). So
	\begin{align}
		  & \hspace*{2.5cm} \norm{  \mathcal{P}_\ep(u)(x,.)-u(x,.) }^2_{H^p(\mathbb{R})}  \nn\\
	 &	=     \int_{\mathbb{R} \setminus \left[-\frac{1}{\epsilon};\frac{1}{\epsilon}\right] }   \exp\left(2(1+\gamma-x)|\omega|^{ {\frac{\alpha+\mu}{2}}} \cos{\frac{\alpha \pi}{4}}\right) \left|\hat{u}(x,\omega)  \right|^2  	(1+\omega^2)^p \exp\left(2(x- \gamma-1)|\omega|^{{\frac{\alpha+\mu}{2}}} \cos{\frac{\alpha \pi}{4}}\right)  d\omega. \nn
		\end{align}
Since $\delta$ {\color{black} satisfies} (\ref{deltachoosing}), we {\color{black} imply} that  $\ep$ satisfies the condition (\ref{lemmamonoas2}). Applying  Lemma \ref{ineqlemma2}  with  $\displaystyle \xi=\frac{\alpha+\mu }{2}>0$, we get 
 $$(1+\omega^2)^p \exp\left(2(x- \gamma-1)|\omega|^{{\frac{\alpha+\mu}{2}}} \cos{\frac{\alpha \pi}{4}}\right) \le \left(1+\epsilon^{-2} \right)^p  \exp\left(2(x- \gamma-1)\ep^{-{\frac{\alpha+\mu }{2}}} \cos{\frac{\alpha \pi}{4}}\right),$$ for all $\displaystyle |\omega| \le \frac{1}{\epsilon}$.  Using the assumption (\ref{assumptionu2}) we have 
$$\int_{\mathbb{R} \setminus \left[-\frac{1}{\epsilon};\frac{1}{\epsilon}\right] }   \exp\left(2(1+\gamma-x)|\omega|^{ {\frac{\alpha+\mu}{2}}} \cos{\frac{\alpha \pi}{4}}\right) \left|\hat{u}(x,\omega)  \right|^2d\omega \le M_2^2.$$ Therefore, 
\begin{align}
		\norm{  \mathcal{P}_\ep(u)(x,.)-u(x,.) }^2_{H^p(\mathbb{R})}  \le   M_2^2 \left(1+\epsilon^{-2} \right)^p  \exp\left(2(x- \gamma-1)\ep^{-{\frac{\alpha+\mu }{2}}} \cos{\frac{\alpha \pi}{4}}\right)          . \label{phpstep2ba}
	\end{align}

\noindent {\bf Sub-step 2c:} Estimate for  the term  $\|u_\ep(x,.)-u(x,.)\|^2_{H^p(\mathbb{R})}$. Combining (\ref{phpetrin2}) and  (\ref{phpstep2aa}), (\ref{phpstep2ba}), we get
	\begin{align}
		 & \|u_\ep(x,.)-u(x,.)\|^2_{H^p(\mathbb{R})} 
		\le  2M_2^2 \left(1+\epsilon^{-2} \right)^p  \exp\left(2(x- \gamma-1)\ep^{-{\frac{\alpha+\mu }{2}}} \cos{\frac{\alpha \pi}{4}}\right) \nn\\
	&  \hspace*{1.7cm}	 + 2K^2 \left(1+ \epsilon^{-2} \right)^p    \exp\left(2x\ep^{-{\frac{\alpha+\mu}{2}}} \cos{\frac{\alpha \pi}{4}}\right)   \int_0^x \mathcal{Z}_2(z)   dz,   \nn
	\end{align} which implies that
	\begin{align}
		\mathcal{Z}_2(x)
		\le &\,\,  2M_2^2 \left(1+\epsilon^{-2} \right)^p \exp\left(-2(\gamma+1)  \ep^{-{\frac{\alpha+\mu }{2}}}  \cos{\frac{\alpha \pi}{4}}\right)    + 2K^2 \left(1+ \epsilon^{-2} \right)^p      \int_0^x \mathcal{Z}_2(z)   dz . \nn
	\end{align}	
Applying the Gronwall's inequality, we obtain 
\begin{align}
\mathcal{Z}_2(x) \le 2M_2^2\left(1+\epsilon^{-2} \right)^p \exp\left(-2 (\gamma+1) \ep^{-{\frac{\alpha+\mu }{2}}}  \cos{\frac{\alpha \pi}{4}}\right) \exp\left( 2K^2 \left(1+ \epsilon^{-2} \right)^p x \right). \label{phpgwb} 
\end{align}	{\color{black} This leads to}
\begin{align}
 & \hspace*{1cm} \|{u}(x,.)-u_\epsilon(x,.)\|^2_{H^p(\mathbb{R})} \le 2M_2^2\left(1+\epsilon^{-2} \right)^p \times \nn\\
 & \times \exp\left(2(x-\gamma-1)  \ep^{-{\frac{\alpha+\mu }{2}}}  \cos{\frac{\alpha \pi}{4}}\right)  \exp\left( 2K^2 \left(1+ \epsilon^{-2} \right)^p   \right). \label{hpstep2b}
\end{align}

\vspace*{0.3cm}
	 
	\noindent {\bf Step 3:} Estimate for  $\|u_{\epsilon}^\delta(x,.)-u(x,.)\|_{H^p(\mathbb{R})}$. It follows from (\ref{etrinhp}), (\ref{phpstep1c}), and (\ref{hpstep2b}) that 
\begin{align}
&   \|u_{\epsilon}^\delta(x,.)-u(x,.)\|_{H^p(\mathbb{R})}^2 \le   16 \left(1+ \epsilon^{-2} \right)^p \exp\left(  2 x \epsilon^{-\frac{\alpha}{2}} \cos\frac{\alpha\pi}{4}\right)     \exp\left(2K^2 \left(1+ \epsilon^{-2} \right)^px \right)  \delta  ^2 \nn\\
 & \hspace*{1cm} + 4M_2^2\left(1+\epsilon^{-2} \right)^p   \exp\left(2(x-\gamma-1)  \ep^{-{\frac{\alpha+\mu }{2}}}  \cos{\frac{\alpha \pi}{4}}\right)  \exp\left( 2K^2 \left(1+ \epsilon^{-2} \right)^p   \right). \nn
\end{align}	
{\color{black}S}o 
\begin{align}
& \|u_{\epsilon}^\delta(x,.)-u(x,.)\|_{H^p(\mathbb{R})}  \le   4 \left(1+ \epsilon^{-2} \right)^{p/2} \exp\left(    x \epsilon^{-\frac{\alpha}{2}} \cos\frac{\alpha\pi}{4}\right)     \exp\left( K^2 \left(1+ \epsilon^{-2} \right)^p  \right)  \delta    \nn\\
 & \hspace*{0.7cm} + 2M_2 \left(1+\epsilon^{-2} \right)^{p/2}   \exp\left( (x-\gamma-1)  \ep^{-{\frac{\alpha+\mu }{2}}}  \cos{\frac{\alpha \pi}{4}}\right)  \exp\left(  K^2 \left(1+ \epsilon^{-2} \right)^p   \right) .\nn
\end{align}\\ Now we consider the estimate (\ref{result2b}).
	 Since  $\delta$ satisfies the condition (\ref{deltachoosing}), we know that  $\displaystyle \frac{1}{\epsilon} >1$. This follows from $\displaystyle 0\le x\cos\frac{\alpha\pi}{4}\le 1$  that
\begin{align}
\left(1+ \epsilon^{-2} \right)^{p/2} & \le \left(e^{\ep^{-2}}\right)^{p/2}, \nn\\
\exp\left(    x \epsilon^{-\frac{\alpha}{2}} \cos\frac{\alpha\pi}{4}\right)     & \le e^{ \epsilon^{-\frac{\alpha}{2}}}, \nn\\
\exp\left( K^2 \left(1+ \epsilon^{-2} \right)^p  \right) &\le e^{ K^2.(2\ep^{-2})^p}, \nn
\end{align}
and  
\begin{align} 
\left(e^{\ep^{-2}}\right)^{p/2}   e^{ \epsilon^{-\frac{\alpha}{2}}}   e^{ K^2 (2\ep^{-2})^p} = e^{\frac{p}{2}\ep^{-2}+  \ep^{-\frac{\alpha}{2}} +  K^22^p\ep^{-2p} } \le e^{ A\ep^{-q}}    \nn
\end{align}	where $\displaystyle q = \max \left\{2;\,\frac{\alpha}{2};\,2p\right\}  $,\,\,  $ \displaystyle A  = \frac{p}{2}  +1+ K^22^p $.  It follows from $\displaystyle  \mu>\max\{4-\alpha;\,4p-\alpha\} $ that $\displaystyle \frac{\alpha+\mu}{2} \ge   \max \left\{2;\,\frac{\alpha}{2};\,2p\right\} = q$. Hence,  $\displaystyle \ep^{-q} \le \ep^{-\frac{\alpha+\mu }{2}}$   and 
\begin{align} 
\left(e^{\ep^{-2}}\right)^{p/2}   e^{ \epsilon^{-\frac{\alpha}{2}}}   e^{ K^2 (2\ep^{-2})^p} \le e^{ A\ep^{-\frac{\alpha+\mu }{2}}} . \nn
\end{align}
Combining {\color{black} the above arguments}, we derive
	\begin{align}
		4 \left(1+ \epsilon^{-2} \right)^{p/2} \exp\left(    x \epsilon^{-\frac{\alpha}{2}} \cos\frac{\alpha\pi}{4}\right)     \exp\left( K^2 \left(1+ \epsilon^{-2} \right)^p  \right)  \delta  
	 \le   4 e^{ A\ep^{-\frac{\alpha+\mu }{2}}} \delta   .\label{theorem2prove24} 
	\end{align}	
	On the other hand,   we also have 
\begin{align}
\left(1+ \epsilon^{-2} \right)^{p/2} &\,\,\le \left(e^{\ep^{-2}}\right)^{p/2}, \nn\\
\exp\left( (x-\gamma-1)  \ep^{-{\frac{\alpha+\mu }{2}}}  \cos{\frac{\alpha \pi}{4}}\right)     &\,\,\le \exp\left(- \gamma \cos{\frac{\alpha \pi}{4}}. \ep^{-{\frac{\alpha+\mu }{2}}}  \right), \nn\\
\exp\left( K^2 \left(1+ \epsilon^{-2} \right)^p  \right) &\,\,\le e^{ K^2(2\ep^{-2})^p}, \nn
\end{align}
and
\begin{align} 
  \left(e^{\ep^{-2}}\right)^{p/2}   \exp\left(- \gamma \cos{\frac{\alpha \pi}{4}}\ep^{-{\frac{\alpha+\mu }{2}}}  \right) e^{ K^2 (2\ep^{-2})^p}  &= e^{ \frac{p}{2}\ep^{-2}+ K^22^p\ep^{-2p} } \exp\left(- \gamma \cos{\frac{\alpha \pi}{4}} \ep^{-{\frac{\alpha+\mu }{2}}}  \right)  \nn\\
& \le e^{ A\ep^{-q} - 2B\ep^{-\frac{\alpha+\mu }{2}}},    \nn
\end{align} where $\displaystyle B=\frac{1}{2}\gamma \cos\frac{\alpha \pi}{4}.$	
From  (\ref{deltachoosing}), we have $\displaystyle \epsilon  < \left(\frac{B}{A}\right)^{\frac{1}{\frac{\alpha+\mu}{2}-q}} $. So 
$
A\ep^{-q} \le  B\ep^{-\frac{\alpha+\mu }{2}}. $
We conclude that the following inequality holds  
\begin{align} 
\left(e^{\ep^{-2}}\right)^{p/2}    \exp\left(- \gamma \cos{\frac{\alpha \pi}{4}} \ep^{-{\frac{\alpha+\mu }{2}}}  \right) e^{ K^2 (2\ep^{-2})^p} &\,\,\le  e^{ -  B\ep^{-\frac{\alpha+\mu }{2}}}. \nn
\end{align}  \\
{\color{black} The above arguments imply that}
	\begin{align}
	 2M_2 \left(1+\epsilon^{-2} \right)^{p/2}   \exp\left( (x-\gamma-1)  \ep^{-{\frac{\alpha+\mu }{2}}}  \cos{\frac{\alpha \pi}{4}}\right)  \exp\left(  K^2 \left(1+ \epsilon^{-2} \right)^p   \right) \le   2M_2   e^{ -  B\ep^{-\frac{\alpha+\mu }{2}}}. \label{theorem2prove5}
 	\end{align}	
	 From (\ref{hpstep2b}), (\ref{theorem2prove24}), and (\ref{theorem2prove5}), we derive
	\begin{align}
		 \|u_{\epsilon}^\delta(x,.)-u(x,.)\| _{H^p(\mathbb{R})}  \le   4 e^{ A\ep^{-\frac{\alpha+\mu }{2}}} \delta    + 2M_2   e^{ -  B\ep^{-\frac{\alpha+\mu }{2}}} . \label{mainresult32}
	\end{align}  
{\color{black} By}  some simple computations, we  get
\begin{align}
	  \|u_{\epsilon}^\delta(x,.)-u(x,.)\|_{H^p(\mathbb{R})}  
		  \le    C_2\delta^{\frac{B}{A+B}}.   \nn
	\end{align} where $\displaystyle C_2= 4+2M_2 $, i.e., the inequality (\ref{result2b}) is proved.	
\end{proof}

\vspace*{0.3cm}

\begin{remark}

To obtain the convergence, we need {\color{black} the strong  assumption} (\ref{assumptionu2}) on the exact solution $u(x,t)$. The techniques are not new and come  from applying    Gronwall's inequality. 
In the next theorem, we will present  a new  way to deal with a weaker assumption (\ref{assumptionu3}) on the exact solution $u(x,t)$.  Indeed, the assumption (\ref{assumptionu3}) is much better than the assumption (\ref{assumptionu2}) of the previous Theorem. 
\end{remark}

\begin{theorem}Assume  that the problem (\ref{exsolution}) has a  unique (exact) solution $u$ such that
	\begin{align}
	\int_{-\infty}^{+\infty}  \exp\Big(2(1+\gamma-x)\Re(k(\omega)) \Big) \left|\hat{u}(x,\omega)  \right|^2 d\omega < M_3^2, \quad 0\le x <1, \label{assumptionu3}
	\end{align} 
	for $M_3>0$. 
	 Let us choose  the regularization parameter  $\displaystyle \ep $  such that
\begin{align}
\displaystyle  \ep < \left[ \frac{\alpha\gamma \cos\frac{\alpha\pi}{4}}{2p} \right]^\frac{2}{\alpha} \label{deltachoosing4}
\end{align} then
	\begin{align}
		&   \|u_{\epsilon}^\delta(x,.)-u(x,.)\| _{H^p(\mathbb{R})} 	\le  C_3 \left(1+\epsilon^{-2} \right)^{p}   \exp\left( x \epsilon^{-\frac{\alpha}{2}} \cos\frac{\alpha\pi}{4}\right) \delta \nn\\
	&\hspace*{1.2cm}	+  C_3 \left(1+\epsilon^{-2} \right)^{p}      \exp\Big(  (x-\gamma-1)\ep^{-{\frac{\alpha}{2}}} \cos{\frac{\alpha \pi}{4}} \Big)   , \label{theorem3ra}
	\end{align}
{\color{black} for all $0\le x <1$} where $$C_3= \max\Big\{ 5e^{K^2};\, 3M_3\left(e^{K^2}+1\right) \Big\} .$$   	 
\end{theorem}


\begin{proof}
Using  the triangle inequality, we have 
\begin{align}
	  \|u_{\epsilon}^\delta(x,.)-u(x,.)\|^2_{H^p(\mathbb{R})} &\le  3\|u_{\epsilon}^\delta(x,.)-{u_\epsilon}(x,.)\|^2_{H^p(\mathbb{R})} +  3\|u_\ep(x,.)-\mathcal{P}_\ep(u)(x,.)\|^2_{H^p(\mathbb{R})}  \nn\\
	&  + 3  \|\mathcal{P}_\ep(u)(x,.)-u(x,.)\|^2_{H^p(\mathbb{R})}   \le   3\left(1+\epsilon^{-2} \right)^p\|u_{\epsilon}^\delta(x,.)-{u_\epsilon}(x,.)\|^2_{L^2(\mathbb{R})} \nn\\
	&  +  3\left(1+\epsilon^{-2} \right)^p \|u_\ep(x,.)-\mathcal{P}_\ep(u)(x,.)\|^2_{L^2(\mathbb{R})}       + 3  \|\mathcal{P}_\ep(u)(x,.)-u(x,.)\|^2_{H^p(\mathbb{R})} . \label{pihpa}
\end{align}	
Moreover, we have 
\begin{align}
  \|u_\ep(x,.)-\mathcal{P}_\ep(u)(x,.)\|^2_{L^2(\mathbb{R})}   &\le  \|u_\ep(x,.)-u(x,.)\|^2_{L^2(\mathbb{R})} + \| \mathcal{P}_\ep(u)(x,.) - u(x,.)\|^2_{L^2(\mathbb{R})} \nn\\
  &  \le \|u_\ep(x,.)-u(x,.)\|^2_{L^2(\mathbb{R})} + \| \mathcal{P}_\ep(u)(x,.) - u(x,.)\|^2_{H^p(\mathbb{R})}. \label{pihpb}
\end{align} Therefore, by combining (\ref{pihpa}) and (\ref{pihpb}), we  obtain the following inequality 
\begin{align}
		 \|u_{\epsilon}^\delta(x,.)-u(x,.)\|^2_{H^p(\mathbb{R})} &\le 3\left(1+\epsilon^{-2} \right)^p\|u_{\epsilon}^\delta(x,.)-{u_\epsilon}(x,.)\|^2_{L^2(\mathbb{R})} \nn\\
		&  +  3\left(1+\epsilon^{-2} \right)^p \|u_\ep(x,.)-u(x,.)\|^2_{L^2(\mathbb{R})}   \nn\\
		& + 3\left[\left(1+\epsilon^{-2} \right)^p+1\right]  \|\mathcal{P}_\ep(u)(x,.)-u(x,.)\|^2_{H^p(\mathbb{R})} . \label{pihpc}
\end{align}	
The  term  $\displaystyle  \|u_{\epsilon}^\delta(x,.)-{u_\epsilon}(x,.)\|^2_{L^2(\mathbb{R})}$  can be similarly estimated as  (\ref{pl2step1e}) 
\begin{align}
	\|u_{\epsilon}^\delta(x,.)-{u_\epsilon}(x,.)\|^2_{L^2(\mathbb{R})} &\,\, \le 8  e^{2K^2 }  \exp\left(2x \epsilon^{-\frac{\alpha}{2}} \cos\frac{\alpha\pi}{4}\right) \delta^2.\label{pihpd}
	\end{align}
Next, we  divide this proof into three steps. The first step is  estimating $\displaystyle \|u_\ep(x,.)-u(x,.)\|^2_{L^2(\mathbb{R})}$, the second step is estimating $\displaystyle \|\mathcal{P}_\ep(u)(x,.)-u(x,.)\|^2_{H^p(\mathbb{R})}$, and the last step is obtaining an estimate for $\displaystyle \|u_{\epsilon}^\delta(x,.)-u(x,.)\|^2_{H^p(\mathbb{R})}$. \\

\noindent {\bf Step 1: }Estimate  $\displaystyle \|u_\ep(x,.)-u(x,.)\|^2_{L^2(\mathbb{R})}$. We have   
\begin{align}
	\|u_\ep(x,.)-u(x,.)\|^2_{L^2(\mathbb{R})}\le &\,\,    2\|u_\ep(x,.)-\mathcal{P}_\ep(u)(x,.)\|^2_{L^2(\mathbb{R})}  + 2  \|\mathcal{P}_\ep(u)(x,.)-u(x,.)\|^2_{L^2(\mathbb{R})}. \label{pihptri2}
\end{align}	In addition, the inequality (\ref{pl2step2aa})  also holds under the assumption \eqref{assumptionu3}
\begin{align}
\,\,& \|u_\ep(x,.)-\mathcal{P}_\ep(u)(x,.)\|^2_{L^2(\mathbb{R})} \le   K^2 \exp\left( 2x\ep^{-{\frac{\alpha}{2}}} \cos{\frac{\alpha \pi}{4}} \right) \int_0^x \mathcal{Z}_1(z)dz, \label{pihpe}
\end{align}
where $\displaystyle \mathcal{Z}_1(z):=\exp\left(-2z\epsilon^{-\frac{\alpha}{2}} \cos{\frac{\alpha \pi}{4}}\right) \| u_{\epsilon} (z,.) - u (z,.) \|_{L^2(\mathbb{R})}^2 \ge 0$ for all $0\le z\le x$.  It follows from (\ref{pl2nopeumu}) that  
\begin{align}
   \mathcal{P}_\ep(u)(x,.)-u(x,.)    =  \mathcal{F}^{-1}\left( -   \hat{u}(x,\omega)   \chi_{\mathbb{R} \setminus \left[-\frac{1}{\epsilon};\frac{1}{\epsilon}\right]}(\omega) \right). \label{pihppeumu}
\end{align} 
Using the assumption (\ref{assumptionu3}), we get
\begin{align}
 & \hspace*{4 cm} \norm{\mathcal{P}_\ep(u)(x,.)-u(x,.)}^2_{L^2(\mathbb{R})} \nn\\
&=    \int_{\mathbb{R} \setminus \left[-\frac{1}{\epsilon};\frac{1}{\epsilon}\right] }  \exp\Big(2(1+\gamma-x)\Re(k(\omega)) \Big) \left|\hat{u}(x,\omega)  \right|^2\exp\Big(2(x-\gamma-1)\Re(k(\omega)) \Big) d\omega \nn\\
&\le   \exp\Big(2(x-\gamma-1)\ep^{-{\frac{\alpha}{2}}} \cos{\frac{\alpha \pi}{4}} \Big) \int_{\mathbb{R} \setminus \left[-\frac{1}{\epsilon};\frac{1}{\epsilon}\right] }  \exp\Big(2(1+\gamma-x)\Re(k(\omega)) \Big) \left|\hat{u}(x,\omega)  \right|^2  d\omega \nn \\
& \le   M_3^2  \exp\Big(2(x-\gamma-1)\ep^{-{\frac{\alpha}{2}}} \cos{\frac{\alpha \pi}{4}} \Big).    \label{pihpf}
\end{align} It follows from (\ref{pihptri2}), (\ref{pihpe}), and (\ref{pihpf}) that
\begin{align}
	\|u_\ep(x,.)-u(x,.)\|^2_{L^2(\mathbb{R})}&\le 2K^2 \exp\left( 2x\ep^{-{\frac{\alpha}{2}}} \cos{\frac{\alpha \pi}{4}} \right) \int_0^x \mathcal{Z}_1(z)dz  \nn\\
	&+ 2  M_3^2  \exp\Big(2(x-\gamma-1)\ep^{-{\frac{\alpha}{2}}} \cos{\frac{\alpha \pi}{4}} \Big). \nn
\end{align}
Multiplying two sides of the above inequality by $\displaystyle \exp\left( 2x\ep^{-{\frac{\alpha}{2}}} \cos{\frac{\alpha \pi}{4}} \right)$, we obtain  
\begin{align}
	\mathcal{Z}_1(x) \le 2  M_3^2  \exp\Big(-2(\gamma+1)\ep^{-{\frac{\alpha}{2}}} \cos{\frac{\alpha \pi}{4}} \Big) +   2K^2   \int_0^x \mathcal{Z}_1(z)dz .\nn
\end{align}
By  applying the Gronwall inequality, we derive
\begin{align}
	\mathcal{Z}_1(x) \le 2  M_3^2  \exp\Big(-2(\gamma+1)\ep^{-{\frac{\alpha}{2}}} \cos{\frac{\alpha \pi}{4}} \Big)e^{2K^2x}, \nn
\end{align}
which leads to  
\begin{align}
	\|u_\ep(x,.)-u(x,.)\|^2_{L^2(\mathbb{R})} \le 2  M_3^2  \exp\Big( 2(x-\gamma-1)\ep^{-{\frac{\alpha}{2}}} \cos{\frac{\alpha \pi}{4}} \Big)e^{2K^2}. \label{pihpg}
\end{align}
\noindent {\bf Step 2: }Estimate $\displaystyle \|\mathcal{P}_\ep(u)(x,.)-u(x,.)\|^2_{H^p(\mathbb{R})}$. It follows from (\ref{pihppeumu}) that 
\begin{align}
		 & \hspace*{4.9cm} \norm{  \mathcal{P}_\ep(u)(x,.)-u(x,.) }^2_{H^p(\mathbb{R})} \nn\\
		  &   =\int_{\mathbb{R} \setminus \left[-\frac{1}{\epsilon};\frac{1}{\epsilon}\right] }   \exp\Big(2(1+\gamma-x)\Re(k(\omega)) \Big) \left|\hat{u}(x,\omega)  \right|^2  	(1+\omega^2)^p \exp\Big(2(x- \gamma-1) \Re(k(\omega)) \Big)  d\omega. \nn
		\end{align}	
Since  $\ep$ satisfies (\ref{deltachoosing4}) and (\ref{lemmamonoas2}) , we   apply  the Lemma \ref{ineqlemma2} for   $\displaystyle \xi=\frac{\alpha }{2}>0$ in order  to obtain
\begin{align}
  \hspace*{0.6cm} (1+\omega^2)^p \exp\Big(2(x- \gamma-1) \Re(k(\omega)) \Big) 
&  =   (1+\omega^2)^p \exp\left(2(x- \gamma-1)|\omega|^{{\frac{\alpha }{2}}} \cos{\frac{\alpha \pi}{4}}\right) \nn\\
& \le  \left(1+\epsilon^{-2} \right)^p  \exp\left(2(x- \gamma-1)\ep^{-{\frac{\alpha  }{2}}} \cos{\frac{\alpha \pi}{4}}\right), \nn
\end{align}
for all $\displaystyle |\omega| \le \frac{1}{\epsilon}$. On the other hand, since  (\ref{assumptionu3}),  we have 
$$\int_{\mathbb{R} \setminus \left[-\frac{1}{\epsilon};\frac{1}{\epsilon}\right] }   \exp\Big(2(1+\gamma-x)\Re(k(\omega)) \Big) \left|\hat{u}(x,\omega)  \right|^2d\omega \le M_3^2.$$ Thus, 
\begin{align}
		\norm{  \mathcal{P}_\ep(u)(x,.)-u(x,.) }^2_{H^p(\mathbb{R})}  \le   M_3^2 \left(1+\epsilon^{-2} \right)^p  \exp\left(2(x- \gamma-1)\ep^{-{\frac{\alpha  }{2}}} \cos{\frac{\alpha \pi}{4}}\right)          . \label{pihph}
	\end{align}

\noindent {\bf Step 3:} Estimate $\displaystyle \|u_{\epsilon}^\delta(x,.)-u(x,.)\|^2_{H^p(\mathbb{R})}$. Combining (\ref{pihpc}), (\ref{pihpd}), (\ref{pihpg}), and (\ref{pihph}), we get 
\begin{align}
	  \|u_{\epsilon}^\delta(x,.)-u(x,.)\|^2_{H^p(\mathbb{R})} 
	&	\le   24 e^{2K^2 }\left(1+\epsilon^{-2} \right)^{ p}   \exp\left(2x \epsilon^{-\frac{\alpha}{2}} \cos\frac{\alpha\pi}{4}\right) \delta^2 \nn\\
	&  	+  6M_3^2e^{2K^2}\left(1+\epsilon^{-2} \right)^{ p}      \exp\Big( 2(x-\gamma-1)\ep^{-{\frac{\alpha}{2}}} \cos{\frac{\alpha \pi}{4}} \Big)  \nn\\
	&   + 3 M_3^2 \left[\left(1+\epsilon^{-2} \right)^{ p} + 1 \right]\left(1+\epsilon^{-2} \right)^{ p}  \exp\left(2(x- \gamma-1)\ep^{-{\frac{\alpha  }{2}}} \cos{\frac{\alpha \pi}{4}}\right)  . \nn
\end{align}   
The inequality  $  \left(1+\epsilon^{-2} \right)^{ p} +1 \le 2\left(1+\epsilon^{-2} \right)^{ p}$ implies that
\begin{align}
 \|u_{\epsilon}^\delta(x,.)-u(x,.)\| _{H^p(\mathbb{R})} 
&		\le  5 e^{ K^2 }\left(1+\epsilon^{-2} \right)^{ p/2}   \exp\left( x \epsilon^{-\frac{\alpha}{2}} \cos\frac{\alpha\pi}{4}\right) \delta  \nn\\
	&  	+  3 M_3 e^{ K^2}\left(1+\epsilon^{-2} \right)^{ p/2}      \exp\Big(  (x-\gamma-1)\ep^{-{\frac{\alpha}{2}}} \cos{\frac{\alpha \pi}{4}} \Big)  \nn\\
	&   + 3 M_3   \left(1+\epsilon^{-2} \right)^{ p}     \exp\left( (x- \gamma-1)\ep^{-{\frac{\alpha  }{2}}} \cos{\frac{\alpha \pi}{4}}\right) . \label{pihpi}
\end{align}
 The inequality (\ref{pihpi})  leads to 
\begin{align}
  \|u_{\epsilon}^\delta(x,.)-u(x,.)\| _{H^p(\mathbb{R})} 
&		\le  C_3 \left(1+\epsilon^{-2} \right)^{p}   \exp\left( x \epsilon^{-\frac{\alpha}{2}} \cos\frac{\alpha\pi}{4}\right) \delta \nn\\
	&  	+  C_3 \left(1+\epsilon^{-2} \right)^{p}      \exp\Big(  (x-\gamma-1)\ep^{-{\frac{\alpha}{2}}} \cos{\frac{\alpha \pi}{4}} \Big)  ,\nn
\end{align} where $C_3= \max\Big\{ 5e^{K^2};\, 3M_3\left(e^{K^2}+1\right) \Big\} $. 
\end{proof}

{\color{black}
	\begin{remark}
		In the case  {of} finite time, the problem can be solved by using the   {Fourier truncation method}. In the future, we will investigate this method to solve the problem. 
	\end{remark}
	
\begin{remark} The boundary conditions (1.3)-(1.4) are given on the left boundary. In the case  {that the boundary conditions are given on the right boundary,} i.e., $$u(1,t)=g_1(t), \quad \textrm{and} \quad  u_x(1,t)=h_1(t).$$ Multiplying the first equation of  (2.9) by $\displaystyle \frac{ \sinh {\Big (k(\omega)(z-x)\Big) }  }{k(\omega)}$    and integrating two sides on $[x;1]$, we derive
	\begin{align}
		\widehat{u}(x,\omega)&= \cosh \Big(k(\omega)(1-x)\Big) \widehat{g_1}(\omega) +  \frac{\sinh \big(k(\omega) (1-x)\Big)}{k(\omega)} \widehat{h_1}(\omega) \nonumber\\
		& -    \int_x^1 \frac{ \sinh {\Big (k(\omega)(z-x)\Big) }  }{k(\omega)} \widehat{f}(z,\omega,u(z,\omega))  dz,  \label{solutionrb} 
	\end{align}
	for $x\ge 0$, $\omega\in\mathbb{R}$. Therefore, the problem corresponding to  {the} right-boundary conditions can be treated  {in} the same way as the problem corresponding to  {the} left-boundary conditions. 
\end{remark}
}

{\color{black}
\section{Numerical example}

\noindent In this section, we present a simple numerical example  to show  {the} efficiency of the method. The  numerical  example  is  implemented  for $t\in [0,  2 \pi]$.  Consider the inverse problem (1.1)-(1.4) according to $g(t)=t^2$, $h(t)=-2t^2$ and  
\begin{align}
f(x,t,u(x,t)) =\,& \frac{u(x,t)}{1+u^2(x,t)} + \tilde{f}(x,t)  \nonumber 
\end{align}
where $$\displaystyle \tilde{f}(x,t)=e^{-2x} \left(    \frac{2 t^{2-\alpha}}{\Gamma(3-\alpha)} - \frac{ t^2}{1+e^{-4x}t^4}  - 4 t^2 \right).$$  Note that $f$  {satisfies} the Lipschitz condition
\begin{align}
	|f(x,t,v_1)-f(x,t,v_2)|\,\,\,\le\,\,\,&  |v_1-v_2|, \quad 0\le x \le 1, \quad t\in \mathbb{R}, \quad v_1, v_2\in \mathbb{R}. \nn
	\end{align}
Moreover, we can easily verify  {that}
$$\displaystyle u(x,t)=e^{-2x}t^2$$ is the exact solution of  {the} problem (1.1)-(1.4). The problem  {is} numerically solved for $x\in (0,1)$ and $t>0$. The noisy data $g^\delta$, $h^\delta$ are given by 
\begin{align}
g^\delta(t)=g(t)\left(1+ \frac{1}{\sqrt{\pi}}~\delta ~ \mathrm{rand}(\mathrm{size(.)}) \right),  \nn\\
h^\delta(t)=h(t)\left(1+ \frac{1}{\sqrt{\pi}}~\delta ~ \mathrm{rand}(\mathrm{size(.)}) \right),  \nn
\end{align} where $\delta$ is the noisy level and $\mathrm{rand}(\mathrm{size(.)})$  {is a} random matrix with elements in  $[-1;1]$. By using the formula (3.41), we can compute the regularized solution $u_\epsilon^\delta$ with respect to the noisy data $g^\delta$, $h^\delta$ and the regularized parameter $\epsilon$ ($\omega_{\max}$).    {The relative error} between the exact solution $u$ and the regularized solution  $u_\epsilon^\delta$  is  {approximated as}
\begin{align}
\mathrm{Error}  = \left( \frac{\sum_{\ell=0}^N |u(x,t_\ell)-u_\epsilon^\delta (x,t_\ell)|^2 }{\sum_{\ell=0}^N |u(x,t_\ell) |^2}  \right)^{1/2} , \nn
\end{align} for fixed space point $x\in (0,1)$. Here
\[
t_{l} = l\Delta t,\quad \Delta t=\frac{2\pi}{N},\quad l=\overline{0,N}.
\]
and we choose $N=512$. 
 Table \eqref{table1} and \eqref{table2} show   {the errors} between  {the} exact solution and the regularized solution  for $\alpha=0.4$ and $\alpha=0.7${.} We can see that the errors are increasing  {when} $\omega_\mathrm{max}$ and $x$  {are} increasing. So the regularized parameter should  {be chosen} larger  {to} get more  {exact results}.  {The figures \ref{figures1} show the regularized solutions for some values of the noisy level $\delta$}. These images show that the errors are  {smaller} when the noisy level $\delta$  {is} smaller.

\begin{table}
		\noindent \begin{centering}
		\caption{  The relative error between $u$ and $u^{\epsilon}_{\omega}$ for $\alpha = 0.4$ with $\omega_{\max1} = 16.9339$, $\omega_{\max2} = 20.9183$, $\omega_{\max3} = 24.9027$ and $\omega_{\max4} = 31.8755$.
			\label{table1}}
			\begin{tabular}{ |l|l|l|l|l|l|l|l|l|l| }
				\hline
				\multicolumn{1}{ |c| }{$x$}  &\multicolumn{4}{|c|}{ $\|u-u_{\alpha}^{\epsilon}\|$ }  \\ \hline
				&  $\omega_{\max1}$ & $\omega_{\max2}$ &  $\omega_{\max3}$ & $\omega_{\max4}$    \\  \hline
				{\tt x = 0.15} & {\tt     0.1378}	  &  {\tt 0.0941}	   & {\tt 0.0756}	  & {\tt 0.0580}	 \\ \hline 
				{\tt x = 0.25} & {\tt     0.1581}	  &  {\tt 0.1015}	  &  {\tt 0.0821} &	    {\tt 0.0639} \\ \hline
				{\tt x = 0.35}&{\tt     0.1472}	  &  {\tt 0.1157}	  &  {\tt 0.0942}	 &   {\tt 0.0758}\\ \hline
				{\tt x = 0.45}&{\tt     0.1869} &	 {\tt   0.1386} &	    {\tt 0.1138} &	 {\tt    0.0966} \\ \hline
				{\tt x = 0.55} & {\tt     0.2160}	 &   {\tt 0.1688}	  &  {\tt 0.1411}	   & {\tt 0.1270}  \\ \hline
				{\tt x = 0.65} & {\tt    0.2396}	  &  {\tt 0.2089}	  &  {\tt 0.1729}	  &  {\tt 0.1631} \\ \hline
				{\tt x = 0.75} & {\tt    0.2729	}  &  {\tt 0.2478}	  &  {\tt 0.2034}	 & {\tt  0.1971
				} \\ \hline
				{\tt x = 0.85} & {\tt    0.2872	} &   {\tt 0.2777}	  &  {\tt 0.2268}	  & {\tt 0.2240
				} \\ \hline
				{\tt x = 0.95} & {\tt    0.3292}	  &  {\tt 0.3022}	 &   {\tt 0.2447}	  &  {\tt 0.2430} \\ \hline
			\end{tabular} \vspace*{0.3cm}
			\par \end{centering}
\end{table}
	\begin{table}
	\caption{The relative error between $u$ and $u^{\epsilon}_{\omega}$ for $\alpha = 0.7$ with $\omega_{\max1} = 16.9339$, $\omega_{\max2} = 20.9183$, $\omega_{\max3} = 24.9027$ and $\omega_{\max4} = 31.8755$.
			\label{table2}}
		\noindent \begin{centering}
			\begin{tabular}{ |l|l|l|l|l|l|l|l|l|l| }
				\hline
				\multicolumn{1}{ |c| }{$x$}  &\multicolumn{4}{|c|}{ $\|u-u_{\alpha}^{\epsilon}\|$ }  \\ \hline
				&  $\omega_{\max1}$ & $\omega_{\max2}$ &  $\omega_{\max3}$ & $\omega_{\max4}$    \\  \hline
				{\tt x = 0.15} & {\tt     0.1358}	  &  {\tt 0.0996}	   & {\tt 0.0805}	  & {\tt 0.0622}	 \\ \hline 
				{\tt x = 0.25} & {\tt         0.1562}	  &  {\tt 0.1244}	  &  {\tt 0.1024}	  &  {\tt 0.0817} \\ \hline
				{\tt x = 0.35}&{\tt         0.2175}	&   {\tt 0.1779	}  &  {\tt 0.1502}	 &   {\tt 0.1255 }\\ \hline
				{\tt x = 0.45}&{\tt         0.2891}	 &   {\tt 0.2620}	 &   {\tt 0.2249}	 &   {\tt 0.1946} \\ \hline
				{\tt x = 0.55} & {\tt         0.3785}	 &   {\tt 0.3583}	 &   {\tt 0.3032	}  &  {\tt 0.2690}  \\ \hline
				{\tt x = 0.65} & {\tt        0.4551}	  &  {\tt 0.4230}	  &  {\tt 0.3585}	  &  {\tt 0.3245} \\ \hline
				{\tt x = 0.75} & {\tt        0.5118	}  &  {\tt 0.4682}	  &  {\tt 0.3916}	  &  {\tt 0.3604} \\ \hline
				{\tt x = 0.85} & {\tt        0.5652}	  &  {\tt 0.4989}	  &  {\tt 0.4153}	  &  {\tt 0.3874} \\ \hline
				{\tt x = 0.95} & {\tt        0.5902	}  &  {\tt 0.5256}	  &  {\tt 0.4372}	 &   {\tt 0.4127} \\ \hline
			\end{tabular}
			\par \end{centering}\vspace*{0.3cm}
	\end{table}
\begin{figure}[h!]
	\begin{center}
		\begin{center}
			\subfigure[$\delta_{1} = 0.5775$]{	\includegraphics[scale=0.42]{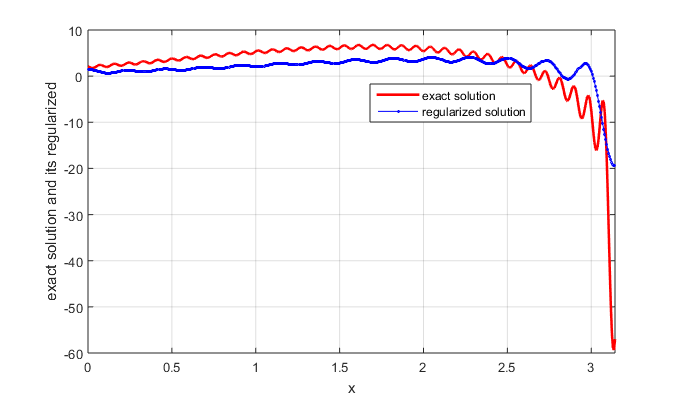}} 
			\subfigure[$\delta_{2} = 0.0543$]{	\includegraphics[scale=0.42]{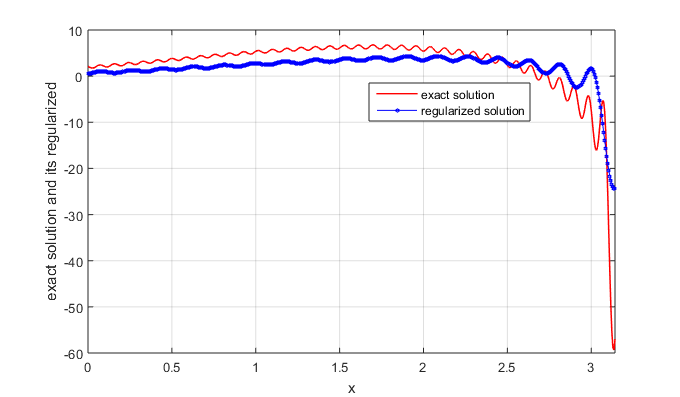} }
			\subfigure[$\delta_{3} = 0.0057$]{	\includegraphics[scale=0.42]{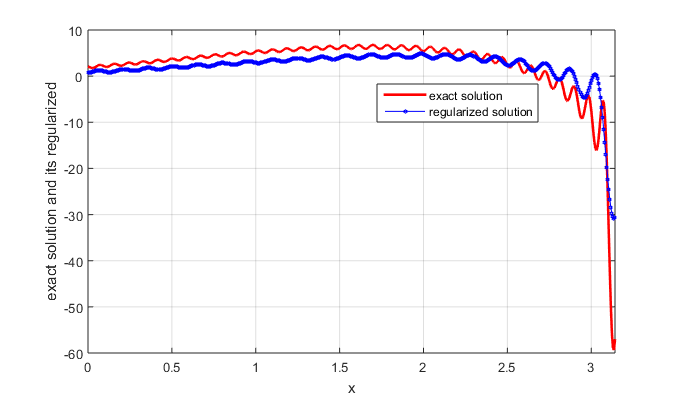} }
			\subfigure[$\delta_{4} = 0.0005$]{	\includegraphics[scale=0.42]{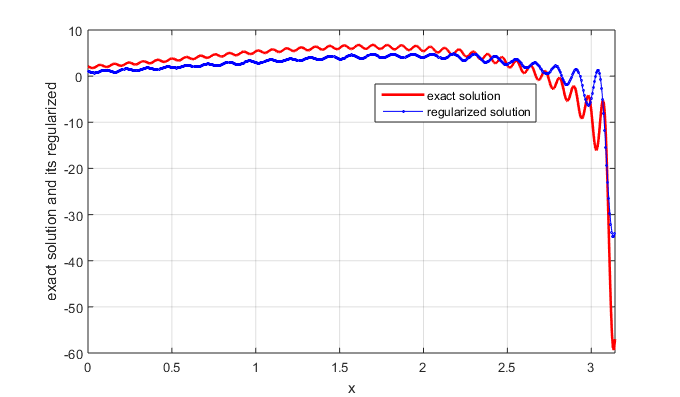} }
		\end{center}
		\caption{A comparison between the exact solution and its computed regularization solution corresponding to $\delta$ .}
		\label{figures1}
	\end{center}	
\end{figure}

\subsection*{Acknowledgements}
This research was supported by Vietnam National Foundation for Science and Technology
Development (NAFOSTED) under grant number 101.02-2019.09. 
The authors would like to thank the editor and two referees
for their  comments and corrections that improved this article.
The author gratefully acknowledge stimulating discussions on numerical results with Le Dinh Long.

\end{document}